\documentclass[a4paper,10pt,reqno]{amsart}

\usepackage{amsmath, amsthm, amssymb}

\usepackage[utf8]{inputenc}
\usepackage{color}
\usepackage[pdftex]{graphicx} 
\usepackage{enumerate}
\usepackage[a4paper]{geometry}
\usepackage{comment}
\usepackage{booktabs}

\definecolor{alertmanzano}{rgb}{0.8,0,0.3}

\newcommand{\alertm}[1]{%
	\marginpar{%
		\ifodd\value{page} \raggedright \else \raggedleft \fi
		\footnotesize{\textcolor{alertmanzano}{#1}}
	}
}

\newtheorem{teo}{\bf Theorem}[section]
\newtheorem{prop}[teo]{\bf Proposition}
\newtheorem{lem}[teo]{\bf Lemma}

\newtheorem{cor}[teo]{\bf Corollary}

\theoremstyle{definition}

\newtheorem{defi}[teo]{\bf Definition}

\newtheorem{remark}[teo]{\bf Remark}

\def\r{\mathbb{R}}
\def\R{\mathbb{R}}

\def\h{\mathbb{H}}
\def\E{\mathbb{E}}
\def\SS{\mathbb{S}}
\def\hr{\mathbb{H}^2\times\mathbb{R}}

\def\la{\langle}
\def\ra{\rangle}

\def\p{\partial}

\def\h{\mathbb{H}}

\title[Minimal surfaces in $\mathbb{H}^2\times\mathbb{R}$ with   finite
  total curvature and embedded ends]{Minimal surfaces in $\mathbb{H}^2\times\mathbb{R}$
  with finite total curvature and embedded ends}

\date{}

\thanks{The first author is partially supported by MICINN-AEI, PID2025-174972NB-I00 and the second author is partially supported by the IMAG–Maria de Maeztu
  grant CEX2020-001105-M / AEI / 10.13039/501100011033 and MINECO/MICINN/FEDER grant no. PID2023-150727NB-I00.}

\author{Jesús Castro-Infantes}
\address{Departamento de Matemática Aplicada a las TIC, Universidad Politécnica de Madrid, Spain}
\email{jesus.castro@upm.es}

\author{Magdalena Rodr\'\i guez}
 \address{Departamento de Geometría y Topología and IMAG (Institute of
   Mathematics of Granada), Universidad de Granada, Spain}
 \email{magdarp@ugr.es}


\begin{document}

\maketitle

\begin{abstract}

In this paper we study the global geometry and classification of complete minimal surfaces with embedded ends and finite total curvature in the product space $\mathbb{H}^2 \times \mathbb{R}$. We describe the structure of embedded ends of such surfaces, constraining the combinatorics of their asymptotic polygons at infinity. Focusing on small total curvature values, we establish a complete classification of the possible asymptotic boundaries for embedded minimal surfaces with total curvature $-4\pi$ and $-6\pi$. In the $-4\pi$ case, we prove the surface is either a horizontal catenoid or simply-connected, with asymptotic boundary belonging to one of four explicit configurations. In the $-6\pi$ case, we show that any embedded example must be simply-connected and classify its possible ideal boundaries. We further construct, via an asymptotic Plateau problem solved by area-minimizing surfaces and successive Schwarz reflections, new 1- and 2-parameter families of (possibly non-embedded) simply-connected minimal surfaces with total curvature $-6\pi$ (and, more generally, $-2(2k-1)\pi$) whose ends are embedded and realize asymptotic configurations not previously known to occur.
\end{abstract}

\section{Introduction}
It is well known that there are no compact minimal surfaces in $\mathbb{R}^3$. The natural counterparts in the non-compact setting are those with finite total curvature. By a classic theorem of Huber~\cite{hu}, such surfaces are conformally equivalent to a compact Riemann surface punctured at a finite number of points, which correspond to the ends of the surface. We recall that the total curvature of a surface $M$ is defined by 
$$C(M)=\int_M K \, dA,$$ 
where $K$ denotes the Gauss curvature of $M$. Since the Gauss map of a minimal surface $M \subset \mathbb{R}^3$ is conformal, its total curvature coincides with the negative of the area of its image under the Gauss map. In particular, when $M$ is complete, its total curvature is either a non-positive multiple of $4\pi$ or minus infinity.

The only complete minimal surfaces in $\mathbb{R}^3$ with vanishing total curvature are planes. The next possible value of the total curvature is $-4\pi$. For this case, Osserman~\cite{oss} proved that the only examples are the catenoid and Enneper's surface; equivalently, these are the only complete minimal surfaces whose Gauss map is injective. On the other hand, Schoen~\cite{sch} proved that the embedded ends of a complete minimal surface in $\mathbb{R}^3$ with finite total curvature are either planar or catenoidal, subsequently characterizing the catenoid as the unique complete minimal surface with finite total curvature and exactly two embedded ends.

Since the turn of the century, the theory of minimal surfaces in the product space $\mathbb{H}^2\times\mathbb{R}$ has undergone extensive development. As in the Euclidean setting, there are no compact minimal surfaces, and the most well-understood examples are those with finite total (intrinsic) curvature. Huber's theorem remains valid in this context: minimal surfaces with finite total curvature are conformally equivalent to a compact Riemann surface minus finitely many points. However, the Gauss map is no longer conformal. Using the Hopf differential (the holomorphic quadratic differential associated with the harmonic height function), Hauswirth and Rosenberg~\cite{hr} proved that the total curvature of a complete minimal surface in $\mathbb{H}^2\times\mathbb{R}$ is a non-positive multiple of $2\pi$, establishing a Jorge-Meeks type formula in this setting.

Hauswirth, Menezes, and the second author~\cite{HMR} characterized complete minimal surfaces in $\mathbb{H}^2\times\mathbb{R}$ with finite total curvature as those with finite topology that are proper, and whose ends are asymptotic to geodesic polygons at the ideal boundary of $\mathbb{H}^2\times\mathbb{R}$ (see Definition~\ref{d:asymptotic} below). 

In Section~\ref{sec:emb}, we show that the asymptotic boundary of a properly embedded minimal surface of $\mathbb{H}^2\times\mathbb{R}$ with finite total curvature can be obtained as the limit of a sequence of embedded polygons at infinity (see Proposition~\ref{prop:limits-polygon} for the precise statement).

The simplest minimal surfaces with finite total curvature in $\mathbb{H}^2\times\mathbb{R}$ are the vertical planes (i.e., cylinders over geodesics of $\mathbb{H}^2$), which are totally geodesic and thus have vanishing total curvature. Hauswirth, Sa Earp, and Toubiana~\cite{hst} proved that these are indeed the only complete minimal surfaces with zero total curvature. The next possible value for the total curvature is $-2\pi$. For this value, Pyo and the second author~\cite{PR} proved an Osserman-type classification: the Scherk minimal graph over an ideal quadrilateral (a vertical graph over an ideal polygonal domain bounded by four ideal geodesics with alternating boundary data $\pm\infty$) is the unique complete minimal surface with total curvature $-2\pi$. 

In Sections~\ref{sec:-4pi} and~\ref{sec:-6pi}, we investigate minimal surfaces realizing the next possible values of total curvature, namely $-4\pi$ and $-6\pi$, and classify the possible asymptotic boundaries for embedded examples. We establish the following classification theorem for the $-4\pi$ case:

\begin{teo}\label{Th:1}
Let $M\subset \mathbb{H}^2\times \mathbb{R}$ be a complete embedded minimal surface with finite total curvature $-4\pi$. Then $M$ is a horizontal catenoid, or a simply-connected example whose asymptotic boundary is one of the polygons described in Theorem \ref{th:configuraciones} (see Figure~\ref{Fig-configuraciones}).
\end{teo}

In the simply-connected case, the only minimal surfaces realizing the asymptotic behavior of case \emph{(1)} are the Scherk graphs~\cite{NR,CR}   (see Section~\ref{sec:pre}) over an ideal geodesic hexagon. This behavior is completely understood, and these solutions form a $2$-parameter family. 

Furthermore, the second author and Pyo~\cite{PR} constructed a $1$-parameter family of surfaces, termed twisted-Scherk examples (see Section~\ref{sec:pre}), whose asymptotic boundaries correspond to cases \emph{(2)} and \emph{(3)}. Whether this family uniquely exhibits these asymptotic behaviors remains an intriguing open question;  a positive answer would complete the classification of minimal surfaces with total curvature $-4\pi$.

For the total curvature value $-6\pi$, we obtain the following classification for the possible asymptotic boundaries:

\begin{teo}\label{Th:2}
Let $M\subset \mathbb{H}^2\times \mathbb{R}$ be a complete embedded minimal surface with finite total curvature $-6\pi$. Then $M$ is simply-connected and its asymptotic boundary is one of the polygons described in Theorem~\ref{th:configuraciones-2} (see Figure~\ref{Fig-configuraciones-2}).
\end{teo}

As before, case \emph{(1)} yields Scherk graphs (see Section~\ref{sec:pre}) over an ideal geodesic octagon, where a \mbox{$4$-parameter} family of surfaces uniquely solves the Dirichlet problem for this boundary configuration. 

Although Theorems~\ref{Th:1} and~\ref{Th:2} are stated for complete embedded minimal surfaces, similar results hold for embedded annular minimal ends of finite total curvature. In case \emph{(3)}, we construct a $1$-parameter family of (possibly non-embedded) surfaces, each of which is simply connected, with an embedded end exhibiting this asymptotic behavior (see Proposition~\ref{prop: Ejemplo 1}). In case \emph{(2)}, we obtain a $2$-parameter family of such surfaces (see Proposition~\ref{prop: Ejemplo 2}). Additionally, we expect that examples corresponding to case \emph{(4)} arise as the conjugate surfaces of our constructed examples for case \emph{(2)}.

The paper is organized as follows. In Section~\ref{sec:pre}, we provide a detailed description of minimal surfaces with finite total curvature and collect key preliminary results. In Section~\ref{sec:emb}, we describe embedded ends with finite total curvature and apply this framework in Section~\ref{sec:low-curvature}  to classify the possible asymptotic boundaries of embedded minimal surfaces with total curvature $-4\pi$ and $-6\pi$. Finally, in Section~\ref{sec:examples}, we construct new families of minimal surfaces with finite total curvature.

\section{Preliminaries}\label{sec:pre}

We consider the cylinder model for $\h^2\times\r=\{(x,y,t): x^2+y^2<1\}$ with the Riemannian metric $g=\frac{4}{(1-x^2-y^2)^2}(dx^2+dy^2)+dt^2$, eventually we will consider a complex parameter $z=x+iy$. We denote by $\pi:\h^2\times\R\to\h^2$ the projection on $\h^2$, that reads as $\pi(x,y,t)=(x,y)$ or $\pi(z,t)=z$.

We consider $\overline{ \h^2\times\r}$, the product compactification for $\h^2\times\R$, denoting by $\R\cup\{\pm\infty\}$ the compactification of $\R$. Then the asymptotic boundary $\partial_\infty(\h^2\times\R)$ of $\h^2\times\R$ consists of the union of $\partial_\infty\h^2\times\R$ (identified in this model with $\mathbb S^1\times\R$) and the horizontal slices $\h^2\times\{\pm \infty\}$. 

We will say that a point $p\in\partial_\infty(\h^2\times\R)$ belongs to the asymptotic boundary $\partial_\infty M$ of a surface $M$, if there exists a divergent sequence of points in $M$ that converges to $p$ in the product compactification. We will use this notion of asymptotic boundary throughout the paper. 

We are going to study the behavior of minimal surfaces in $\h^2\times\R$ with finite total curvature, advancing in the classification of those with small total curvature. We say that a surface $M$ has finite total curvature if its Gauss curvature is integrable as a function in $M$. In that case the total curvature of $M$ is defined as $\int_M K$.

We start by recalling some previously known results. We refer to~\cite{hr,hnst,HMR} for more details. We will deal with the ends of a minimal surface in $\h^2\times \r$ with finite total curvature, which are asymptotic to an admissible polygon at infinity (see Theorem~\ref{th:ctf}). Let us first explain what this means.

\begin{defi}\label{d:poligono}
We call \emph{polygon at infinity} any connected,
closed curve in $\partial_\infty(\h^2\times\r)$ formed by a finite number of geodesics. 
We say that a polygon at infinity is \emph{admissible} if it is composed of the same number of complete geodesics in $\h^2\times\{+\infty\}$ as in $\h^2\times\{-\infty\}$, none of them consecutive, together with the vertical lines in $\partial_\infty\h^2\times \r$ joining their endpoints.
We say that an admissible polygon $\mathcal{P}$ is embedded if there exists a one-to-one correspondence from $\SS^1$ to $\mathcal{P}$.
\end{defi}

We observe that we allow repetitions in the sense that when we go along the polygon at infinity we can pass twice (or more times) through the same geodesic; we count the geodesic with multiplicity. Even we could go along the polygon at infinity several times.

We deduce from Huber's Theorem~\cite{hu} that a minimal surface $M\subset\hr$ with finite total curvature has finite topology; in particular, its ends are annular. 

\begin{defi}\label{d:asymptotic}
 Let $\E$ be an annulus and $X:\E\to \h^2\times\r$ be a proper immersion such that $E=X(\E)$ is an end of a minimal surface $M\subset\hr$ with $\partial_\infty E=\mathcal P$ an admissible polygon at infinity. 
 \begin{itemize}
     \item We say that $E_i$ is a \emph{vertical sheet} of $E$ for $t_0>0$ if $E_i$ intersects any horizontal plane above $\{t=t_0\}$ (resp. below $\{t=-t_0\}$) and there exists a connected component $U_i$ of $X^{-1}(E\cap\{t>t_0\})$ (resp. $X^{-1}(E\cap\{t<-t_0\})$) such that $E_i=X(U_i)$. We will abbreviate by saying that $E_i$ is a vertical sheet of $E$ when it is a vertical sheet for some sufficiently large $t_0$.
     \item We say that $E_i$ is a \emph{horizontal sheet} of $E$ for a horocylinder $\mathcal H$ if it intersects any horocylinder contained in $\mathcal H$ and there exists a connected component $U_i$ of $X^{-1}(E\cap\mathcal H)$ such that $E_i=X(U_i)$. We will abbreviate by saying that $E_i$ is a horizontal sheet of $E$ when it is a horizontal sheet for some horocylinder.
     \item We say that the end $E$ is \emph{asymptotic to} $\mathcal{P}$ (recall $\mathcal P=\partial_\infty E$) if there exists $t_0>0$ large enough such that, for any vertical sheet $E_i$ of $E$ for $t_0$, there exists a complete geodesic $\alpha_i\subset \pi(\mathcal P)$ so that $\partial_\infty E_i\subset \partial_\infty(\alpha_i\times\R)$.
     \item We say that $E$ is \emph{asymptotic with multiplicity one} to $\mathcal P$ if for any horizontal (resp. vertical) ideal geodesic  $\Gamma\subset\mathcal P$ there is a unique vertical (resp. horizontal) sheet of $E$ for some large $t_0>0$ (resp. for some horocylinder $\mathcal{H}$) containing $\Gamma$ in its asymptotic boundary.     
 \end{itemize}
\end{defi}
In the above definition, if $\mathcal P$ has horizontal intersecting geodesics and $E$ has vertical sheets forming a corner when approaching the intersection points, then $E$ is not asymptotic to $\mathcal P$. 

We observe that if the end is embedded then all the vertical sheets (resp. horizontal sheets) are disjoint.
In this work we will usually deal with properly embedded minimal surfaces with finite total curvature, so each end of the surface will be asymptotic with multiplicity one to an admissible polygon at infinity ${\mathcal P}$ without transverse self-intersections. In this case, the end $E\subset M$ being asymptotic to ${\mathcal P}$ reduces to the condition $\partial_\infty E={\mathcal P}$.

The following theorem describes minimal surfaces with finite total
curvature.

\begin{teo}\cite{hr,hnst, HMR}\label{th:ctf}
  Let $M$ be a complete, minimal surface immersed in
  $\h^2\times \r$ with finite total curvature. Then:
  \begin{enumerate}
  \item $M$ is conformally a closed Riemann surface $\mathbb{M}$
    punctured in a finite number of points $p_1,\dots, p_r$ corresponding to the 
    ends of $M$.
  \item $M$ is proper.
  \item The angle function $\nu=\la N, \frac{\p}{\p t}\ra$ converges
    uniformly to zero on each end $p_i$.
  \item Any end $p_i$ of $M$ is asymptotic to an
    admissible polygon at infinity composed of $m_i+1$ geodesics in
    $\h^2\times\{+\infty\}$ and $m_i+1$ geodesics in
    $\h^2\times\{-\infty\}$ (counting with multiplicity), joined by $2(m_i+1)$ vertical straight lines in
    $\partial_\infty \h^2\times\R$, for an integer $m_i\geq 0$. 
    \item\label{item} Any vertical or horizontal sheet of $E$ can be written outside a compact set as a horizontal graph over a vertical plane $\alpha\times\R$, being $\alpha$ the projection of a horizontal geodesic in $\mathcal P$, converging to zero asymptotically. When we say that it is a horizontal graph we refer to both possible meanings: a killing graph generated by horizontal translations and in geodesic directions orthogonal to the plane.
  \item The total curvature of $M$ is given by
    \begin{equation}\label{eq}
      \int_M K=2\pi\left(2 -2 g-2r-\sum\limits_{i=1}^{r}m_i
      \right) ,
    \end{equation}
    where $g$ is the genus of $\mathbb M$, $r$ the numbers of ends
    of $M$, and $m_i\geq 0$ is the number associated to each end $p_i$ (explained above).
  \end{enumerate}
\end{teo}

We also have the following characterization for minimal surfaces with
finite total curvature:

\begin{teo}\cite{HMR}\label{th:caracterizacion-ctf}
  A complete minimal surface in $\h^2\times\R$ has finite total
  curvature if and only if it is proper, it has finite topology and
  each of its ends is asymptotic to an admissible polygon at infinity.
\end{teo}

In what follows, we describe the known families of surfaces with finite total curvature; see Table~\ref{tab:examples} for a summary of the general picture.

\medskip 

\emph{\bf Vertical planes}: Given a complete geodesic $\alpha\subset\hr$, we call $\alpha\times\r$ a vertical plane. It is totally geodesic, so it has zero total curvature. Hauswirth, Sa Earp and Toubiana~\cite{hst} proved that vertical planes are the only complete minimal surfaces with vanishing total curvature in $\hr$. 

\medskip 

\emph{\bf Scherk graphs}: Let $\Omega$ be an ideal geodesic polygon with $2k$ edges, $k\geq 2$, and prescribed boundary values $+\infty$ and $-\infty$ alternately. Assume $\Omega$ satisfies the Jenkins-Serrin conditions (i.e. the flux conditions obtained in~\cite{NR,CR} for the Jenkins-Serrin problem to admit a solution). Hence there exists a unique (up to vertical translations) minimal vertical graph over $\Omega$ with the prescribed boundary values, called Scherk graph (over $\Omega$) for similarity with the corresponding (non-complete) examples in the Euclidean space. 

Collin and Rosenberg~\cite{CR} used Fatou’s convergence theorem in the
Gauss-Bonnet formula to calculate the total curvature of this
example: it has total curvature $-2\pi(k-1)$. 
We then observe that Scherk graphs provide examples for any possible non-zero value of the total curvature.

For any fixed $k$, there is a $2k-4$ parameter family of Scherk graphs up to isometries. In particular, there is a unique Scherk graph over an ideal quadrilateral (i.e. for $k=2$).  In~\cite{PR} Pyo and the second author proved that it corresponds to the unique minimal surface with total curvature $-2\pi$. 

\medskip 

\emph{\bf Horizontal catenoid}: In~\cite{MoR,p}, it was constructed a 1-parameter family of minimal surfaces with genus zero and two planar ends (asymptotic to vertical planes, see Definition~\ref{def:ends} below), see Figure~\ref{Fig-CTF-EJEMPLOS}. They have total curvature $-4\pi$. Any one of these examples is symmetric with respect to three planes: one horizontal (i.e. a copy of $\h^2$ at some fixed height) and two vertical ones. Any half surface on one side of the horizontal plane of symmetry is a vertical graph; we say that it is a \emph{vertical bigraph}. 

These examples are called \emph{horizontal catenoids} as they share the same properties as the catenoid in $\R^3$: they are annuli with finite total curvature $-4\pi$. We add "horizontal" in the name to distinguish them from the rotationally invariant catenoids with vertical axis constructed by Nelli and Rosenberg in~\cite{NR}. 

Schoen proved that the catenoid is the only minimal surface in $\R^3$ with finite total curvature and two embedded ends. In $\hr$ we have a $1$-parameter family of non-isometric horizontal catenoids. The parameter can be considered as the distance between the asymptotic vertical planes, going from zero (getting as a limit the vertical plane) to the distance between two opposite horizontal geodesics in the asymptotic boundary of the Scherk graph over an ideal quadrilateral.  When approaching this maximal value for the parameter, the upper half of the catenoids converge to the Scherk graph when fixing the upper points of the necks at a fixed point, say the origin; so the catenoids split into the union of two Scherk graphs.

Hauswirth, Nelli, Sa Earp and Toubiana~\cite{hnst} proved a Schoen-type theorem in this setting: the
horizontal catenoids are the only complete minimal surfaces in $\hr$
with finite total curvature and two planar ends. (Asymptotic for them implied that the ends are Killing
graphs over a vertical plane, in particular embedded; but this extra
hypothesis can be removed thanks to the results proved in~\cite{HMR}.) 

\medskip 

\emph{\bf $k$-noids}: As a generalization of the horizontal catenoids, a family of properly embedded minimal surfaces with genus zero and $k\geq 3$ planar ends were obtained in~\cite{MoR,p}, see Figure~\ref{Fig-CTF-EJEMPLOS}. All these examples are vertical bigraphs and have total curvature $-4\pi(k-1)$. 

They were named for similarity with the $k$-noids in $\R^3$ constructed by Jorge and Meeks in~\cite{JM}: they have the same
topology, they have finite total curvature and they are vertical bigraphs. In $\R^3$, each end of a $k$-noid
is asymptotic to half a catenoid, so they are non-embedded as far as $k\geq 3$. 

In~\cite{MoR} it was proved that, for any $k$, there is a $(2k-3)$-parameter family of $k$-noids.

The analogous constant mean curvature (CMC) examples in the symmetric case  with mean curvature $0<H\leq \frac 1 2$ were constructed in \cite{CMR}.

\medskip 

\emph{\bf Genus one $k$-noids}: In~\cite{CM}, it was constructed a 1-parameter family of minimal surfaces with genus one and $k\geq 3$ planar ends, see Figure~\ref{Fig-CTF-EJEMPLOS}. All these examples are vertical bigraphs and they are symmetric with respect to $k$ vertical planes; they look like $k$-noids with a hole in the middle. These examples  have total curvature $-4\pi k$. 

In~\cite{MMR}, Mart\'\i n, Mazzeo and the second author obtained using a different approach surfaces with similar properties for a large number of ends, not necessarily disposed symmetrically. They proved that, for any large $k$, there is a $(2k-3)$-parameter family of genus one $k$-noids. They also constructed examples of minimal surfaces with finite total curvature, planar ends (maybe not disposed cyclically) and any positive genus. The number of ends of these examples is arbitrarily large depending on the genus. It was also proved that, for any fixed genus $g\geq 1$ and up to isometries, every example lies in a $(2k-3)$-parameter family, being $k$ the number of ends. All these examples have total curvature $-4\pi (g+k-1)$. 

The ends of all known examples with positive genus are asymptotic to vertical planes. 
Using the Maximum Principle with vertical planes we deduce that vertical planes are the only examples with one planar end. By the Shoen-type Theorem explained above, the only example with two planar ends is the horizontal catenoid. Hence, the minimum number of ends of a genus $g\geq 1$ minimal $k$-noid is three. In this sense, the construction by the first author and Manzano is sharp.   

While the number of ends of a finite total curvature minimal surface in $\mathbb{R}^3$ is bounded by its genus~\cite{MPR}, no such upper bound holds in $\mathbb{H}^2 \times \mathbb{R}$~\cite{MMR}. This raises the interesting question of finding the minimal number of ends for such surfaces in $\mathbb{H}^2 \times \mathbb{R}$.  Minimal surfaces with finite total curvature, positive genus and one or two non-planar ends are expected. 

For CMC surfaces with mean curvature $0<H\leq \frac 1 2$, the first author and Santiago in~\cite{CS} constructed analogous examples: genus one surfaces with ends asymptotic to vertical cylinders.  In this context, it is possible to construct CMC surfaces with   $k\geq 2$ ends.

\medskip 

\emph{\bf Twisted-Scherk examples}: In~\cite{PR} the first examples of minimal surfaces with finite total curvature that are neither vertical graphs nor vertical bigraphs were constructed. They are the union of two vertical graphs along a vertical geodesic in their boundary, see Figure~\ref{Fig-CTF-EJEMPLOS}. One of these vertical graphs is obtained as a solution to a Jenkins-Serrin problem (with $\pm\infty$ boundary values disposed alternately) over a domain with one interior vertex and $2k+1$ ideal vertices,  $k\geq 1$. The other vertical graph is obtained from the previous one by reflection about the vertical geodesic in its boundary. These examples are simply-connected and they have total curvature $-4\pi k$. 

We observe that some of these examples could be non-embedded, as the domain is not necessarily convex (the angle at the interior vertex could be bigger than $\pi$). 

We also observe that in the case $k=1$, the domain must be symmetric with respect to a geodesic passing through the interior vertex (with endpoint one of the ideal vertices) in order to satisfy the Jenkins-Serrin conditions obtained in~\cite{NR,CR}. Up to an isometry we can assume that the interior vertex is placed at origin and the geodesic (of symmetry) is $\{y=0\}$. Once we fix an ideal vertex in $\{y>0\}$ the third one in $\{y<0\}$ is determined. Hence there is a 1-parameter family of twisted-Scherk examples with total curvature $-4\pi$, all of them containing a horizontal geodesic.

\begin{figure}[h]
\begin{center}
\includegraphics[height=8cm]{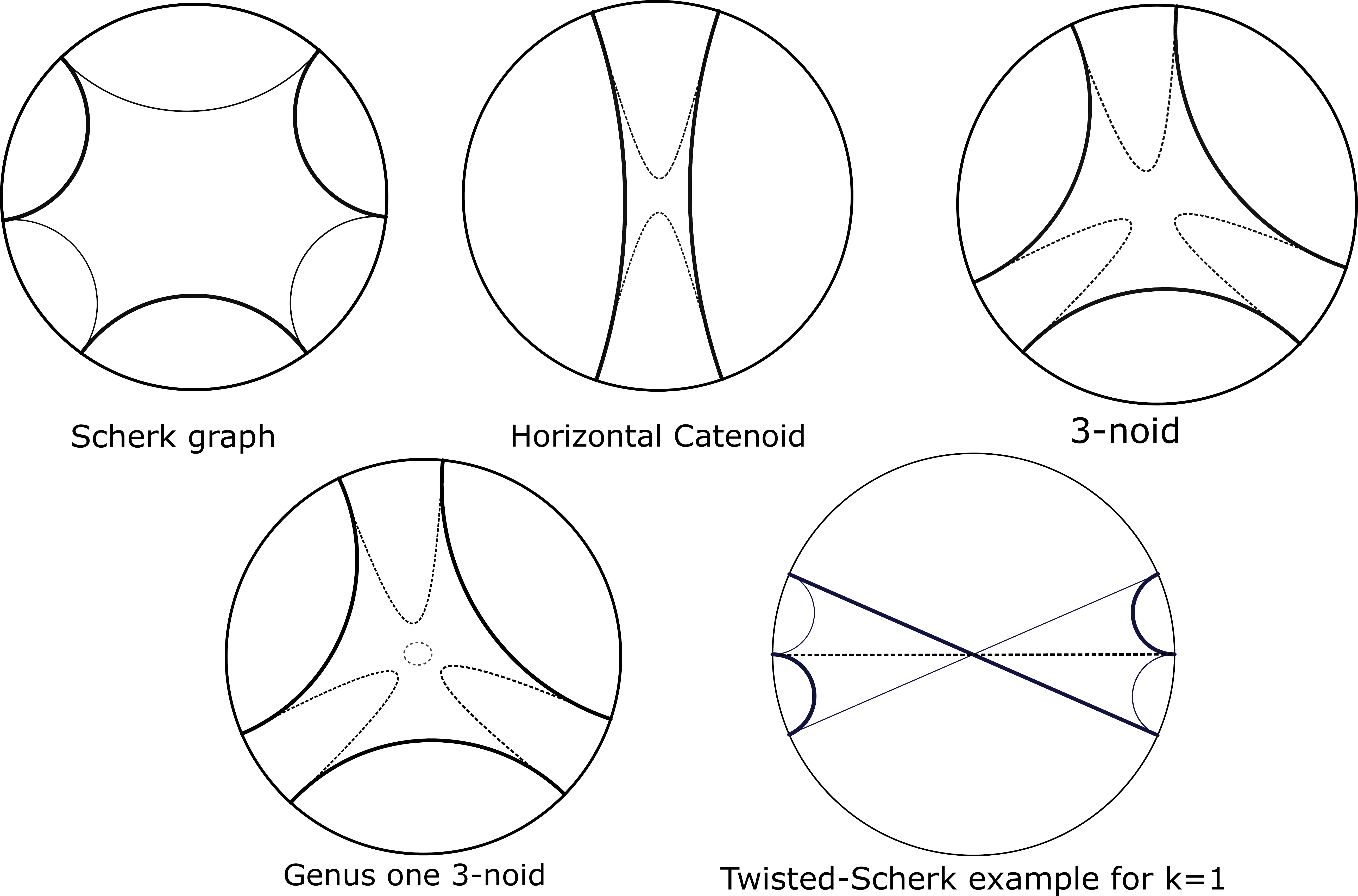}
\end{center}
\caption{Projection of some minimal surfaces with finite total curvature in $\mathbb H^2\times\R$: Scherk graphs, horizontal catenoids, $k$-noids, genus one $k$-noids, and twisted-Scherk surfaces. The solid lines represent geodesic in $\h^2\times\{\pm\infty\}$ and the dotted lines represent symmetry curves of the minimal surfaces or horizontal geodesics contained in the minimal surfaces.}
	\label{Fig-CTF-EJEMPLOS}
\end{figure}

\begin{table}[htbp]
\centering
\label{tab:examples}
\renewcommand{\arraystretch}{1.35}
\begin{tabular}{|p{3.cm}|c|c|p{1.9cm}|c|c|}
\hline
\textbf{Surface} & \textbf{Genus $g$} & \textbf{Ends $r$} & \textbf{$m_i$} & \textbf{Total curvature} & \textbf{Reference} \\
\hline
Vertical plane & $0$ & $1$ & $m_1=0$ & $0$ & [12] \\
\hline
Scherk graph over ideal $2k$-gon ($k\geq 2$) & $0$ & $1$ & $m_1=k-1$ & $-2\pi(k-1)$ & [5, 20] \\
\hline
Horizontal catenoid & $0$ & $2$ & $m_1=m_2=0$ & $-4\pi$ & [19, 23] \\
\hline
$k$-noid  & $0$ & $k\geq 3$ & $m_i=0,\ \forall i$ & $-4\pi(k-1)$ & [19, 23] \\
\hline
Genus one $k$-noid  & $1$ & $k\geq3$ & $m_i=0,\ \forall i$ & $-4\pi k$ & [1] \\
\hline
Genus $g$  $k$-noid ($g\geq 1$) & $g$ & $k>>1$ & $m_i=0,\ \forall i$ & $-4\pi(g+k-1)$ & [15] \\
\hline
Twisted-Scherk example ($k\geq 1$) & $0$ & $1$ & $m_1=2k$ & $-4\pi k$ & [24] \\
\hline
New example Prop.~5.1 and Remark~\ref{remark: Ejemplo k} ($k\geq 2$)  & $0$ & $1$ & $m_1=2k-1$  & $-2(2k-1)\pi $ &  Prop.~5.1, Figure~\ref{Fig-Superficies} \\
\hline
New example Prop.~5.4 & $0$ & $1$ & $m_1=3$ & $-6\pi$ &  Prop.~5.4, Figure~\ref{Fig-Superficies}\\
\hline
\end{tabular}
\caption{Known examples of complete minimal surfaces with finite total curvature in $\mathbb{H}^2\times\mathbb{R}$.}
\end{table}

\newpage 

Next we recall the known classification results for (complete) minimal surfaces with finite total curvature. We also include a non-existence result for examples featuring at least one planar end and one Scherk-type end, defined as follows:

\newpage

\begin{defi}\label{def:ends}\mbox{}
    \begin{itemize}
    \item An end is said to be \emph{planar} if its associated $m_i$ is zero, i.e. it is asymptotic with multiplicity one to the asymptotic boundary of a vertical plane.
	\item  We say that an admissible polygon at infinity is  \emph{Scherk-type} if its vertical projection is injective and it is the boundary of a convex ideal geodesic polygon with $2k$ edges, for some $k\geq 2$. We say that an end is \emph{Scherk-type} if it is asymptotic with multiplicity~one to a Scherk-type admissible polygon.
	\end{itemize}
      \end{defi}

\begin{remark}\label{rem:zero}
We deduce from item~{\it(\ref{item})} in Theorem~\ref{th:ctf} that any planar or Scherk-type end is embedded.  
We also observe that a Scherk-type end is not necessarily  asymptotic to the asymptotic boundary of a Scherk graph because the end has not necessarily zero vertical flux, see equation~\eqref{eq:flux} below.
\end{remark}

In~\cite{hst} a classification for vertical planes as the only minimal surfaces in $\hr$ with zero total curvature was obtained. We include an alternative proof using Theorem~\ref{th:ctf}: 
From~\eqref{eq} we deduce that if $M$ has zero total curvature, then its genus must be zero and it only has one planar end. (The surface cannot be compact so $r\geq 1$.) Then $M$ is simply-connected and $\partial_\infty M= \partial_\infty(\alpha\times\r)$, for some complete geodesic $\alpha\subset\h^2$. Now we conclude using the maximum principle with vertical planes that $M=\alpha\times\r$.

The next value for the total curvature is $-2\pi$. In~\cite{PR} it was proved that the only minimal surface with total curvature $-2\pi$ is the Scherk graph over an ideal quadrilateral. We sketch the idea of the proof. By~\eqref{eq}, a minimal surface $M$ with total curvature $-2\pi$ must be simply-connected with $m_1=1$. Since the end cannot be planar, 
then either the horizontal geodesics in $\h^2\times\{+\infty\}$ or $\h^2\times\{-\infty\}$ (or both) do not intersect. By item (5) in Theorem~\ref{th:ctf}, the vertical sheets of $M$ are horizontal graphs. That allows us to start the Alexandrov reflection principle using horizontal planes and conclude that $M$ must be a vertical graph, and then a Scherk graph.

The idea behind this proof (used in~\cite[Theorem 6 and 7]{HMR}) yields the following non-existence proposition.

\begin{prop}\label{th:Alexandrov1}
Let us consider $p\geq 1$ vertical planes $\Pi_1,\ldots,\Pi_p$ and $q\geq 1$ Scherk-type polygons at infinity
$\mathcal{P}_1,\ldots, \mathcal{P}_q$. Let $\Omega\subset\h^2$ be the convex hull of their vertical projection, i.e. the convex hull of $(\cup_i \pi(\Pi_i))\cup(\cup_j \pi(\mathcal{P}_j))$ in $\h^2$. Suppose that the ends can be ordered cyclically, more specifically: $\pi(\Pi_i)\subset\partial\Omega$ for every $i\in\{1,\ldots, p\}$ and, for any $j\in\{1,\ldots,q\}$, there exists a horizontal geodesic $\gamma_j\subset\mathcal{P}_j\cap(\h^2\times\{-\infty\})$ such that $\pi(\mathcal{P}_j-\gamma_j)\subset\partial\Omega$. Then there is no properly embedded minimal surface
in $\h^2\times\R$ with asymptotic boundary $\partial_\infty\Pi_1 \cup \ldots \cup \partial_\infty\Pi_p\cup \mathcal{P}_1 \cup \ldots \cup \mathcal{P}_q$.
\end{prop}
\begin{remark}
 Similar argument produces the same conclusion when the geodesics $\gamma_j$ are contained in $\h^2\times\{+\infty\}$. 

Note that the vertical planes $\Pi_i$ and/or the Scherk-type polygons at infinity $\mathcal{P}_j$ in Proposition~\ref{th:Alexandrov1} can share a vertical straight line in $\partial_\infty\mathbb{H}^2\times \mathbb{R}$.   
\end{remark}

\begin{proof}[Proof of Proposition~\ref{th:Alexandrov1}] Assume there exists such a surface~$M$. We deduce using the maximum principle with vertical planes that $M$ is contained in $\Omega\times\mathbb{R}$. 

Since $M$ is properly embedded, $M\cap \{t>t_0\}$ consists of a finite number of disjoint components (vertical sheets), each one of them a horizontal graph by item $(5)$ in Theorem~\ref{th:ctf}. Let $E_i$ be one such component and $\alpha_i$  be the complete geodesic such that $\partial_\infty E_i\subset\partial_\infty(\alpha_i\times\r)$. We call $E_i^s$ the reflected image of $E_i$ with respect to $\{t=t_0\}$. We translate horizontally $E_i^s$ to the connected component of the exterior of $\Omega\times\r$ bounded by $\alpha_i\times\r$ so that the translated $E_i^s$ does not intersect $M$. When going back to its original position, the translated $E_i^s$ cannot intersect $M$ by the maximum principle. Hence $M\cap \{t<t_0\}$ lies on one side of $E_i^s$. Since this argument follows for any vertical sheet, we conclude that  $M\cap \{t<t_0\}$ lies on one side (``the interior") of the symmetry of $M\cap \{t>t_0\}$ with respect to $\{t=t_0\}$. We can thus apply the Alexandrov reflection principle, leading to the conclusion that $M$ is either a vertical graph or a vertical bigraph. In both cases, we reach a contradiction by using (respectively) that $M$ has a planar or a Scherk-type end. 
\end{proof}

\begin{remark}\label{rem:scherk}
The argument in the proof of Proposition~\ref{th:Alexandrov1} can be applied in the case $p=0$ and $q=1$ (resp. $p>0$ and $q=0$) obtaining in this case that the surface is a vertical graph (resp. a vertical bigraph).  This case corresponds to~\cite[Theorem 7]{HMR} (resp.~\cite[Theorem 6]{HMR}).
\end{remark}

We conclude this section by obtaining a flux condition for the ends of a minimal surface with finite total curvature, relative to an equilibrium between the horizontal geodesics in $\h^2\times\{+\infty\}$ and $\h^2\times\{-\infty\}$.

Let $E$ be an annular end of a minimal surface. Let $c\subset E$ be a representative of the end and call $\eta$ the outward-pointing unit conormal along $c$ (i.e. if we parameterize the end on a punctured disk, the one not pointing to the puncture). We define the vertical flux of $E$ as 
\[
F(E)= \int_c \langle \eta,\partial_t\rangle .
\]
This definition does not depend on $c$. It is clear that if $M$ is a minimal disk, then $F(\partial M)=0$. 

Suppose that the end $E$ is asymptotic to an admissible polygon at infinity $\mathcal P$. We deduce from item $(3)$ in Theorem~\ref{th:ctf}, by taking a sequence of representatives of the end converging to  $\mathcal P$, that 
\begin{equation}\label{eq:flux}
F(E)=\alpha(\mathcal P)-\beta(\mathcal P),    
\end{equation}
where $\alpha(\mathcal P)$ (resp. $\beta(\mathcal P)$) denote the length in $\h^2$ of the vertical projection of the
horizontal geodesics of $\mathcal P$ in $\h^2\times\{+\infty\}$ (resp. $\h^2\times\{-\infty\}$) outside some disjoint horodisks at the vertices of $\pi(\mathcal P)$. This is one of the Jenkins-Serrin conditions in the case $M$ is a vertical graph.

\section{Minimal surfaces with finite total curvature and embedded ends}
\label{sec:emb}

From now on we assume that $M$ is a minimal end properly embedded in
$\h^2\times \r$ with finite total curvature. We know that $M$ is
asymptotic to an admissible polygon at infinity $\mathcal{P}$ which
consists of a finite number geodesic
$\bar l_1,\bar l_3,\dots, \bar l_{2k-1}$ in
$\h^2\times\left\lbrace +\infty\right\rbrace$, with $k\geq 2$, the
same number of geodesics $\bar l_2,\bar l_4,\dots,\bar l_{2k}$ in
$\h^2\times\left\lbrace -\infty\right\rbrace$, and the corresponding
vertical geodesics in $\partial_\infty \h^2\times \r$ joining
their endpoints. Let $l_i=\pi(\bar l_i)$ be their vertical projection on $\h^2$
and assume they are ordered following a fix orientation in the polygon at infinity 
$\mathcal{P}$.  

\begin{lem}\label{l1}
  Let $M\subset \h^2\times \r$ be an embedded minimal end with finite
  total curvature and $\partial_\infty M=\mathcal{P}$, an admissible polygon at infinity as above. Then no two horizontal geodesics in $\mathcal P$ can intersect transversely. In particular (using the notation above) if $l_i$ and $l_j$ intersect transversely then $i$ and $j$ cannot have the same parity.
\end{lem}
\begin{proof}
 We use the notation above. Suppose that $\bar l_1$ and $\bar l_3$ intersect transversely (the remaining cases follow similarly). By Theorem~\ref{th:ctf} we know that there are two vertical sheets $E_1,E_2$ of $M$ (see definition~\ref{d:asymptotic}) with $\partial_\infty E_1\subset\partial_\infty(l_1\times (0,+\infty))$ and $\partial_\infty E_2\subset\partial_\infty(l_3\times(0,+\infty))$, which are horizontal graphs on the vertical planes $l_1\times \R$ and $l_3\times\R$ respectively. Since these vertical planes intersect transversely, then $E_1$ and $E_2$ must intersect, a contradiction with the embeddedness of $M$.
\end{proof}

Lemma~\ref{l1} follows from the fact that no two vertical sheets of
$M$ can intersect. Now let us use that two horizontal sheets cannot
intersect either.
Denote by $p_i,p_{i+1}\in \partial _\infty \h^2$ the ideal endpoints
of $l_i$ using cyclical notation (i.e. $p_0=p_{2k}$ and
$p_{2k+1}=p_1$) and let $v_i=\{p_i\}\times\r$ be the vertical straight
lines contained in $\mathcal P$.

\begin{lem}\label{l2}
  Let $M\subset \h^2\times \r$ be an embedded minimal end with finite
  total curvature and let $\mathcal{P}$ be its asymptotic boundary.
  Suppose there exist two pairs of consecutive geodesics in $\pi(\mathcal P)$ sharing a common endpoint in $\partial_\infty\h^2$ such that any pair defines a different horizontal sheet of $M$ near the vertical line over the common endpoint. (In particular, the sub-indices of the four geodesics are all different.) Let $\gamma\subset\h^2$ be a
  geodesic near the common endpoint that intersects them. Thus, with the notation above, the four
  geodesics can be called $l_i,l_{i+\varepsilon_i}, l_j,l_{j+\varepsilon_j}$, with $\varepsilon_i,\varepsilon_j\in\{1,-1\}$ in such a way that they intersect $\gamma$ in one of the following orders (non-strictly, in the sense that two or
  more of the four geodesics can coincide):
  \begin{enumerate}
  \item $l_i, l_j, l_{i+\varepsilon_i}, l_{j+\varepsilon_j}$, with
    $(-1)^{i+j}=1$ (i.e. $i$ and $j$ have the same parity).
  \item $l_i,l_{i+\varepsilon_i}, l_j, l_{j+\varepsilon_j}$. 
  \end{enumerate}
\end{lem}

\begin{proof}
  We call $l_i,l_{i+\varepsilon_i}, l_j$ and $l_{j+\varepsilon_j}$, for some $\varepsilon_i,\varepsilon_j\in\{1,-1\}$ the four geodesics in the hypothesis. Let $q\in\partial_\infty\h^2$ be their common endpoint. 
We will prove that we can rename the geodesics so that they intersect $\gamma$ following one of the desired orders. 
  
  By Theorem~\ref{th:ctf} there are two different horizontal sheets for some horocylinder $\mathcal H$ at $q$, $E_i$ and $E_j$, such that $\partial_\infty E_s=(\bar l_s\cap\mathcal{H})\cup(\{q\}\times\R)\cup(\bar l_{s+\varepsilon_s}\cap\mathcal{H})$ for $s\in\{i,j\}$. We know that $E_i, E_j$ are horizontal graphs over any vertical plane containing $\{q\}\times\R$ at its asymptotic boundary. 
  
  We can assume that the geodesic $\gamma\subset\h^2$ intersects the  geodesics $l_i,l_{i+\varepsilon_i}, l_j, l_{j+\varepsilon_j}$ inside $\pi(\mathcal H )$. Thus the vertical plane $\gamma\times\R$ intersects $E_s$, with $s\in\{i,j\}$, along a diverging curve (contained in $E_s$) traveling from $\bar l_s$ to $\bar l_{s+\varepsilon_s}$, that projects horizontally onto the vertical line $(l_s\cap \gamma)\times\r$. By embeddedness
  these two curves cannot intersect, so one lies on one side of the other when considered on the vertical plane $\gamma\times\R$, see Figure~\ref{Fig-Lemma-curves}. Hence the only possibilities are those described in the lemma, up to possibly rename the geodesics.
  \end{proof}

\begin{figure}[h]\label{Fig-Lemma-curves}
\begin{center}
\includegraphics[height=8cm]{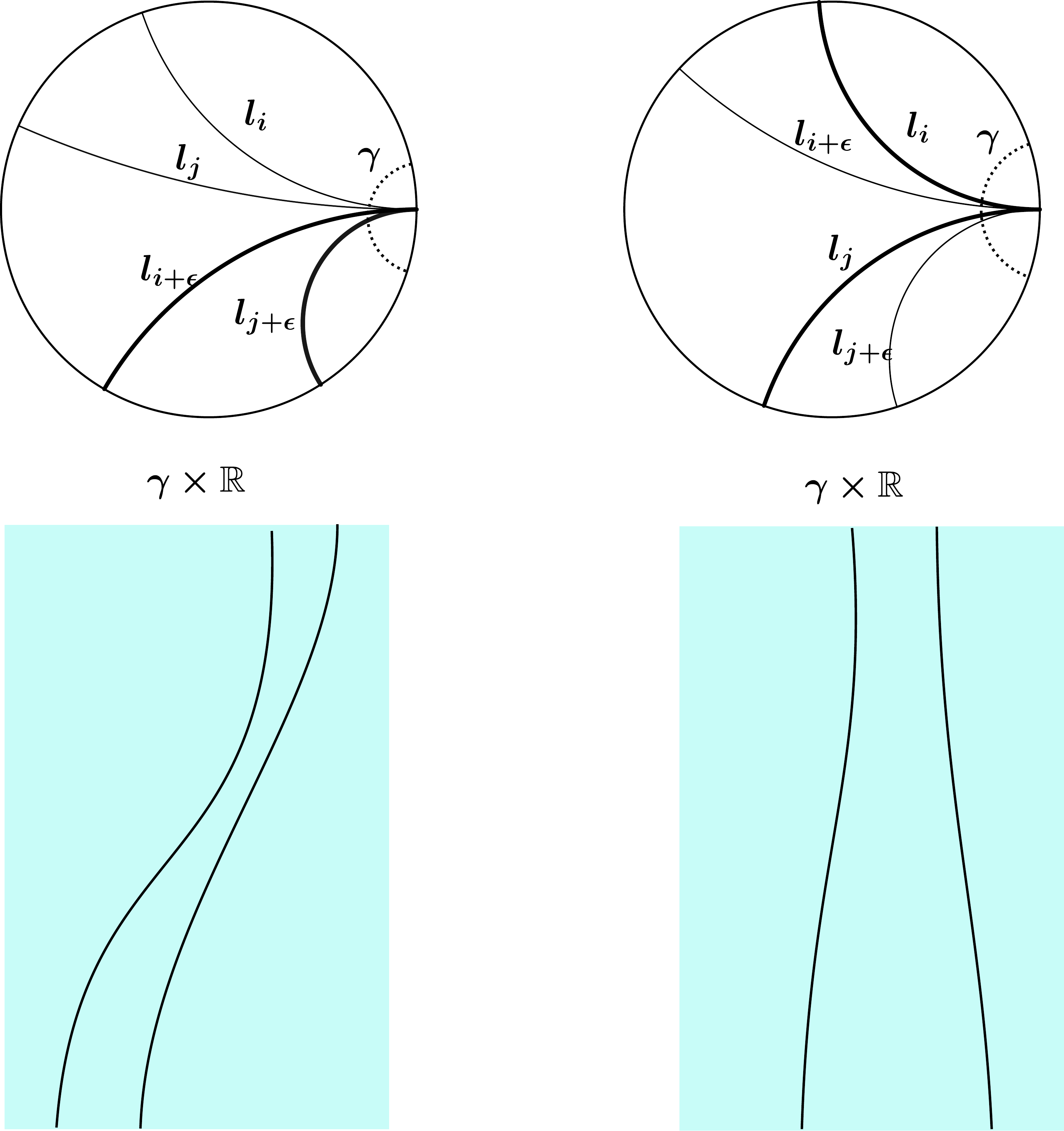}
\end{center}
\caption{Two particular examples of Lemma~\ref{l2} lying in cases {\it (1)} (left) and {\it (2)} (right). Up: the geodesics of $\pi(\mathcal P)$ sharing a common vertex. Down: the diverging curves in $\gamma\times\mathbb R$.}
	\label{fig:HorizontalSheets}
\end{figure}

\begin{remark}\label{rem:order}
In Lemma~\ref{l2} we have really proved that if we have two horizontal sheets with a common vertical line in its asymptotic boundary, then we intersect one horizontal sheet before the other (by embeddedness) when we go around a circumference near $\partial_\infty\h^2$. Fix the positive orientation in the circumference and suppose we intersect first $E_i$. We then say that $E_j$ is \emph{``above"} $E_i$ or $E_i$ is \emph{``below"} $E_j$. 

That defines an order in the geodesics as well. Thus in Lemma~\ref{l2} we can write $l_i\leq l_j\leq l_{i+\varepsilon_i}\leq l_{j+\varepsilon_j}$ in $(1)$ and $l_i\leq l_{i+\varepsilon_i}\leq l_j\leq l_{j+\varepsilon_j}$ in $(2)$. We remark that even in the case when the geodesics $l_i,l_j$ coincide in ($1$), we will not write $l_j\leq l_i\leq l_{i+\varepsilon_i}\leq l_{j+\varepsilon_j}$ as we will take into account the order of the horizontal sheets to order the geodesics. 
\end{remark}

\begin{lem}\label{l:1side}
    Let $M\subset\h^2\times\mathbb R$ be an embedded minimal end with finite total curvature and let $\mathcal P$ be its asymptotic boundary. If two or more consecutive geodesics in $\pi(\mathcal P)$ coincide, then one of the following cases hold:
    \begin{enumerate}
        \item $k=1$ and $l_1=l_2$ (i.e. the end is planar). 
        \item $k>1$ and there exist different $i,j\in\{1,\dots,2k\}$ such that $l_{i-1}\neq l_i=l_{i+1}=\dots=l_j\neq l_{j+1}$ with $l_{i-1}, l_{j+1}$ in different components of $\h^2-\{l_i\}$ (in particular, $M$ cannot project to one side of $l_i$).
    \end{enumerate}
\end{lem}

\begin{proof}
    If $k=1$ the only possibility is $l_1=l_2$, and item {\it (1)} hold. Let us then assume $k>1$. 
    
    Since $M$ is embedded, $\pi(\mathcal{P})$ cannot be a geodesic. Hence there exist $i,j\in\{1,\dots,2k\}$ such that $l_{i-1}\neq l_i=\dots=l_j\neq l_{j+1}$. 
    We can assume $l_i=\{y=0\}$, $l_{i-1}\subset\{y<0\}$, and $(1,0)$ is the common endpoint of $l_i,l_{i-1}$. Let us prove that $l_{j+1}\subset\{y>0\}$.

    If $i,j$ have different parity then $l_{i-1},l_{j+1}$ share  $(1,0)$ as common endpoint. By Lemma~\ref{l2} the only possibility is $l_{i-1}\leq l_i\leq l_j\leq l_{j+1}$. Hence $l_{j+1}\subset\{y>0\}$.

    Let us suppose $i,j$ have the same parity. If we call $E_i$ (resp. $E_j$) the horizontal sheet defined by $l_{i-1},l_i$ (resp. $l_{j-1}=l_j$), then $E_j$ is ``above" $E_i$. Let us call $M_i$ (resp. $M_j$) the vertical sheet having $\bar l_i$ (resp. $\bar l_j$) in its asymptotic boundary. It intersects $E_i$ (resp. $E_j$) in an open set. By continuity, we find $M_j$ first when coming from $\{y>0\}$. Now call $E'_i$ (resp. $E'_j$) the horizontal sheet defined by $l_i, l_{i+1}$ (resp. $l_j, l_{j+1}$). It coincides with $M_i$ (resp. $M_j$) in an open set. Therefore we get that $E'_i$ must be "above" $E'_j$, and the lemma follows.
\end{proof}

As an example of item {\it (1)} in Lemma~\ref{l:1side} we have any planar end (for example, any end of a horizontal catenoid). In Section~\ref{sec:examples} we construct new examples that are in  situation {\it (2)}, see Proposition~\ref{prop: Ejemplo 1} and Figure~\ref{Fig-Varios-LADOS} for $\alpha=\frac \pi k$, $k\geq 2$.


\begin{figure}[h]\label{new-ends}
	\begin{center}
		\includegraphics[height=4cm]{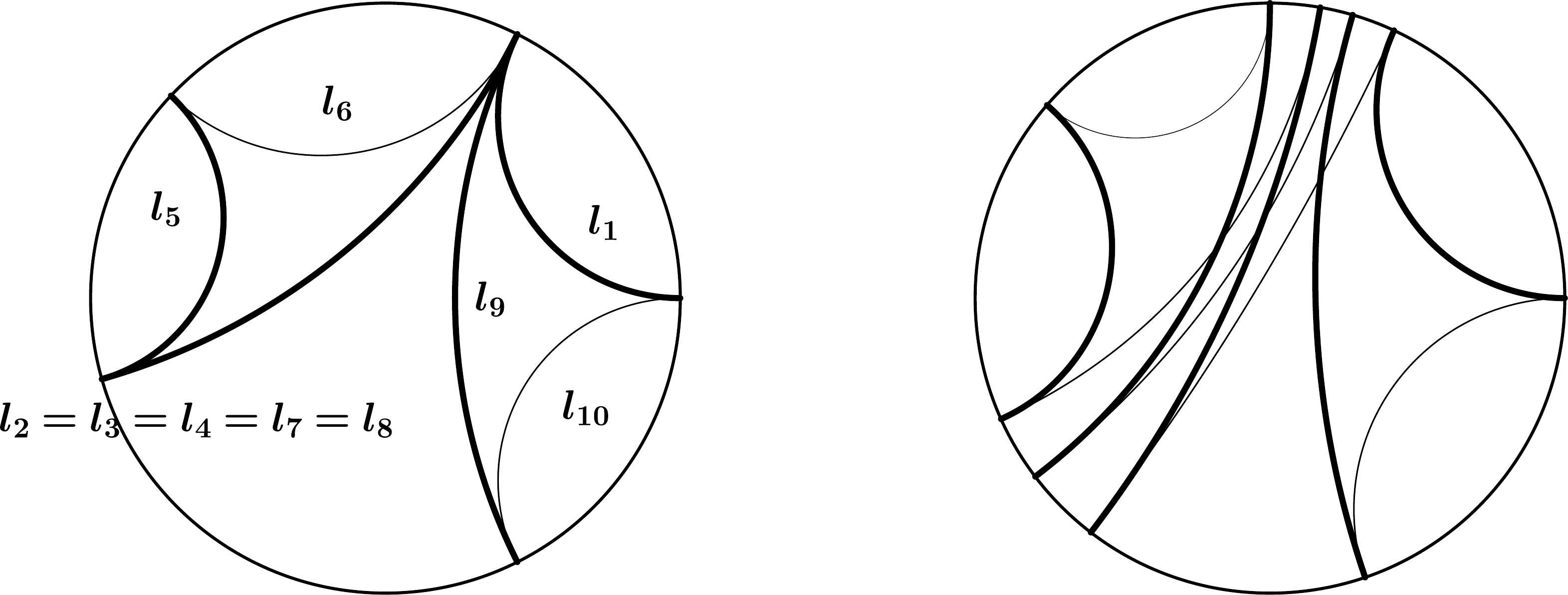}
        \end{center}
	\caption{Possible ends satisfying item {\it (2)}  in Lemma~\ref{l:1side} (left)  and a generic polygon approaching it (right).}
	\end{figure}

In the asymptotic boundary of the twisted-Scherk example (see Section~\ref{sec:pre}) with $k=1$ and second ideal vertex $(0,1)$, non-consecutive geodesics coincide. This surface can be thought as a limit of other twisted-Scherk examples corresponding to the choice of second ideal vertex $e^{\theta_n\pi i}$, with $\theta_n<\frac\pi 2$ converging to $\frac\pi 2$. We are going to prove that this is not a special case: With the exception of the asymptotic boundary of a vertical plane (item {\it (1)} in Lemma~\ref{l:1side}), every asymptotic boundary of a properly embedded minimal surface of $\h^2\times\r$ with finite total curvature, can be obtained as a limit of generic polygons at infinity defined as follows:

\begin{defi}\label{d:generic}
We say that a polygon at infinity is \emph{generic} if it is admissible, has more than two horizontal geodesics, and exactly two geodesics of $\pi(\mathcal{P})$ meet at each of its ideal vertices.
\end{defi}

\begin{prop}\label{prop:limits-polygon}
    Let  $\mathcal P$ be an admissible polygon at infinity containing more than two horizontal geodesics (i.e. it is not the asymptotic boundary of a vertical plane). Then there exists a sequence of generic polygons at infinity $\{\mathcal P_n\}_n$ with the same number of horizontal geodesic as $\mathcal P$, that converges to~$\mathcal P$. Moreover, if $\mathcal{P}$ is the asymptotic boundary of a non-planar embedded minimal end in $\h^2\times\r$, then the generic polygons at infinity $\mathcal P_n$ can be chosen to be embedded. 
\end{prop}

\begin{proof}
If $\mathcal P$ is generic, we take $\mathcal P_n=\mathcal P$ for any $n$. So let us assume that $\mathcal P$ is not generic. 

We use the notation above. It is clear that if $l_p,l_{p+1}$ share their common endpoint, that can be assumed to be $(1,0)$,  with $l_j,l_{j+1}$ then we can change $l_j,l_{j+1}$ slightly by moving their common endpoint to $e^{i\delta_n}$ for $\delta_n>0$ converging to zero and $\delta_n$ small enough so that there is no vertex of ${\mathcal P}$ between $(1,0)$ and $e^{i\delta_n}$. We can follow a similar argument for any possible common vertex to construct the desired polygons ${\mathcal P}_n$.

Let us now assume that $\mathcal{P}$ is the asymptotic boundary of a non-planar embedded minimal end (and non-generic). Now we must be more careful in constructing the generic polygons at infinity so that they are embedded.

Since $\mathcal{P}$ is not generic, there exist more than two geodesics of $\pi(\mathcal{P})$ arriving to the same vertex in $\partial_\infty\h^2$, say $(1,0)$. There must be an even number (at least four) of them, and they can be grouped in pairs so that each pair defines a horizontal sheet with $\{(1,0)\}\times\r$ in its asymptotic boundary. Since $M$ is embedded, these horizontal sheets are ordered (see Remark~\ref{rem:order}). If $l_j,l_{j+1}$ are the geodesics defining the horizontal sheet ``above" the others, we change $l_j,l_{j+1}$ so that their new common endpoint is $e^{i\delta_n}$ for some $\delta_n\to 0$ as above. 
If the new polygon at infinity is generic, we are done. Otherwise, we repeat the same argument with the geodesics arriving to the same vertex. Since there are a finite number of geodesics in $\pi(\mathcal{P})$, after a finite number of steps we obtain the desired  generic embedded polygons at infinity.
\end{proof}

\begin{cor}\label{cor2}
  Let $M\subset \h^2\times \r$ be an embedded minimal end with finite
  total curvature and suppose that it only has four associated
  geodesics $l_1, l_2, l_3, l_4$ (i.e. $m=1$). Then the four geodesics
  form a quadrilateral, and the end is Scherk-type.
\end{cor}
  \begin{proof}
 We call $\mathcal{P}$ the asymptotic boundary of $M$. By Proposition~\ref{prop:limits-polygon}, there is a sequence of embedded generic polygons at infinity $\mathcal{P}_n$, any of them with four horizontal geodesics, converging to $\mathcal{P}$.   By
    Lemma~\ref{l1}, the four horizontal geodesics of $\mathcal{P}_n$ form a quadrilateral, thus the same happens for $\mathcal{P}$. 
  \end{proof}

\section{Minimal surfaces with low total  curvature}\label{sec:low-curvature}

\subsection{Minimal surfaces with total curvature $-4\pi$ in
  $\h^2\times \r$}\label{sec:-4pi}

We will first prove that a complete minimal surface with total
curvature $-4\pi$ cannot have positive genus.

\begin{prop}\label{prop:genus}
  Let $M\subset \h^2\times\r$ be a complete minimal surface with total
  curvature $-4\pi$. Then $M$ has genus zero.
\end{prop}

\begin{proof}
  By~\eqref{eq} we know that $2g+2r+\sum\limits_{i=1}^{r}m_i=4$, where $g$ is the genus of $M$, $r$~the number of its ends, and $m_i$ is the number associated to each end related to the number of horizontal geodesics contained in its asymptotic boundary. Since
  $M$ cannot be compact we know that $r\geq 1$, and then $g\leq 1$.
  
  Suppose by contradiction $g=1$. Then it holds necessarily $r=1$ and
  $m_1=0$. This is, $M$ has only one end and it is asymptotic to a
  vertical plane. By the maximum principle using vertical planes we
  obtain that $M$ is a vertical plane, a contradiction.
\end{proof}

Let $M\subset \h^2\times\r$ be a complete minimal surface with total
curvature $-4\pi$. From equation~\eqref{eq} and the previous proposition we
obtain $2r+\sum\limits_{i=1}^{r}m_i=4$. Thus there are two
possibilities:
\begin{itemize}
  \item either $r=2$ and $m_1=m_2=0$, i.e. $M$ is an annulus with
    planar ends; or
    \item $r=1$ and $m_1=2$, i.e. $M$ is simply-connected and
      asymptotic to a polygon at infinity with three horizontal
      geodesics in $\h^2\times\{+\infty\}$ and three geodesics in
      $\h^2\times\{-\infty\}$.
\end{itemize}

In the first case, Remark~\ref{rem:zero} says that the planar ends are embedded, and we get that $M$ must be a horizontal catenoid by the  Schoen-type Theorem~\cite{hnst,HMR}.

\medskip

Therefore from now on we assume that $M$ is a minimal disk in $\h^2\times \r$ with total curvature $-4\pi$. In order to obtain a classification result, we restrict our study to the case where $M$ is embedded (then properly embedded). 

We use the notation introduced in Section~\ref{sec:emb}: $M$ is asymptotic to an admissible polygon at infinity
$\mathcal{P}$, 
$\bar l_1\cup\bar l_3\cup\bar l_5=\mathcal{P}\cap (\h^2\times\left\lbrace +\infty\right\rbrace)$, 
$\bar l_2\cup\bar l_4\cup\bar l_6=\mathcal{P}\cap (\h^2\times\left\lbrace -\infty\right\rbrace)$,
$l_i=\pi(\bar l_i)$ has endpoints $p_i,p_{i+1}\in \partial _\infty \h^2$ (using cyclical notation $p_0=p_6$ and $p_7=p_1$), and  $v_i=\{p_i\}\times\r$. 
We assume they are ordered following a fix orientation in the asymptotic polygon at infinity, thus
\[
\mathcal{P}=v_1\cup\bar l_1\cup v_2\cup\bar l_2\cup v_3\cup\bar l_3\cup v_4\cup\bar l_4\cup v_5\cup\bar l_5\cup v_6\cup\bar l_6 .
\]

Up to an isometry
we can assume that $p_1=(1,0)\equiv 1$. Let $\theta_j\in [0, 2\pi)$
satisfy $p_j=e^{i\theta_j}$ for any $j$. We observe that these
$\theta_j$ are not necessarily monotonically ordered (see Figure~\ref{Fig-CTF-EJEMPLOS}) and that two or three of them could coincide. We
consider the orientation in $\mathcal P$ so that
$\theta_2\leq \theta_6$. The following theorem describes the possible
asymptotic boundaries for $M$.

\begin{teo}\label{th:configuraciones}
  Let $M\subset \h^2\times \r$ be a complete embedded minimal surface
  with finite total curvature $-4\pi$ and one end. With the notation
  above and up to an isometry and a possible reordering, one of the following cases must hold:
  \begin{enumerate}
  \item $0=\theta_1<\theta_2<...<\theta_6<2\pi$ (the ideal points
    $p_i$ are cyclically ordered). In this case, $M$ must be a Scherk graph over an ideal geodesic hexagon.
  \item   $0=\theta_1<\theta_2<\theta_5<\theta_4<\theta_3<\theta_6<2\pi$.
  \item    $0=\theta_1<\theta_2<\theta_5<\theta_4<\theta_3=\theta_6<2\pi$.
  \item    $0=\theta_1<\theta_2=\theta_5<\theta_4<\theta_3=\theta_6<2\pi$.
  \end{enumerate}
\end{teo}

\begin{figure}[htb]
	\begin{center}
		\includegraphics[height=10cm]{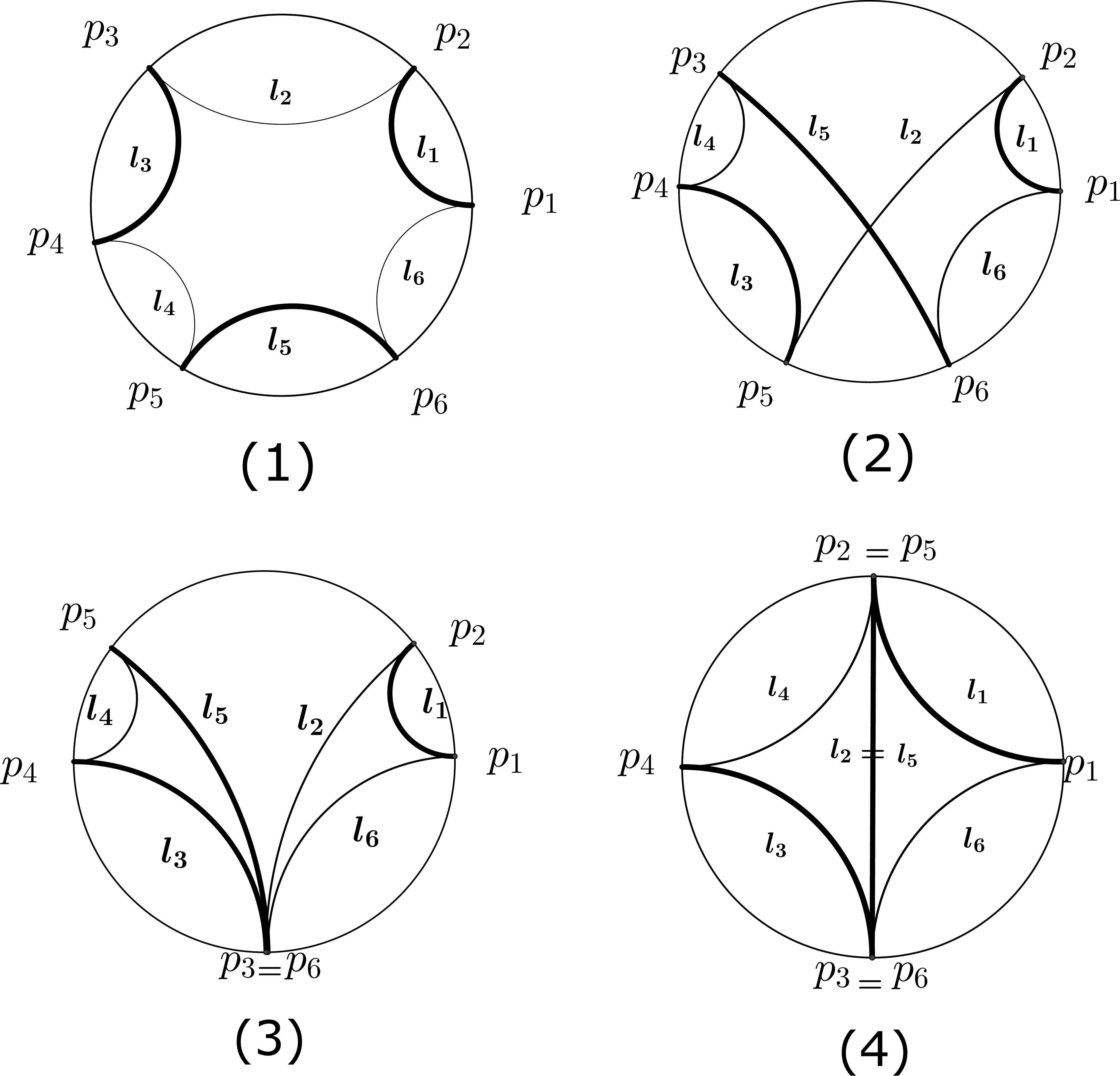}
	\end{center}
	\caption{Projections of the possibles configurations of the
          admissible polygon at infinity in
          Theorem~\ref{th:configuraciones}.}
	\label{Fig-configuraciones}
\end{figure}
 
Complete Scherk minimal graphs correspond to the first case. For the
second and fourth case we have twisted-Scherk Examples. No example for
the third case is known yet (if it exists).

\begin{proof}
By Proposition~\ref{prop:limits-polygon} there exists an embedded generic polygon at infinity $\mathcal P_n$ arbitrarily close to $\mathcal P$. Let us analyze the possibilities for $\mathcal P_n$. By Lemma~\ref{l1},
either no geodesics in $\pi(\mathcal P_n)$ intersect and case \it{(1)} follows; or two geodesics with different parity intersect, and we get \it{(2)}.

The possible limit of polygons at infinity in case \it{(1)} lies again in case \it{(1)}. The possible limit for $\mathcal P_n$ in case \it{(2)} lies in cases \it{(2)}, \it{(3)} or \it{(4)}. 

Finally, we observe that, by Remark~\ref{rem:scherk}, the only possible surfaces corresponding to case $(1)$ are Scherk graphs. This proves the theorem.
\end{proof}

Let us describe the possible configurations using two parameters $\alpha,\beta\in(0,\frac\pi 2)$. Up to an isometry, we can fix $p_1=(1,0)$, $p_4=(-1,0)$, and $p_6=(0,-1)$; i.e. $\theta_1=0$, $\theta_4=\pi$, and $\theta_6=-\frac\pi 2$. 
Then $\theta_2\in(0,\frac{\pi}{2})$ is our first choice and $\theta_3\in(-\pi,-\frac\pi 2)$ our second one. (We can always assume $\theta_2\leq -\theta_6$ up to an isometry.)  Now $\theta_5$ is determined by the flux condition~\eqref{eq:flux}. Geometrically, we place a horodisk $H_6$ at $p_6$ and a horodisk $H_j$ at $p_j$ tangent to $H_{j-1}$ for any $1\leq j\leq 4$, see Figure~\ref{Fig-parametros}. 
Finally, we consider the horodisk $H_5$ tangent to both $H_4$ and $H_6$ that determines $p_5=H_5\cap\partial_\infty\h^2$.

\begin{figure}[htb]
	\begin{center}
		\includegraphics[height=6cm]{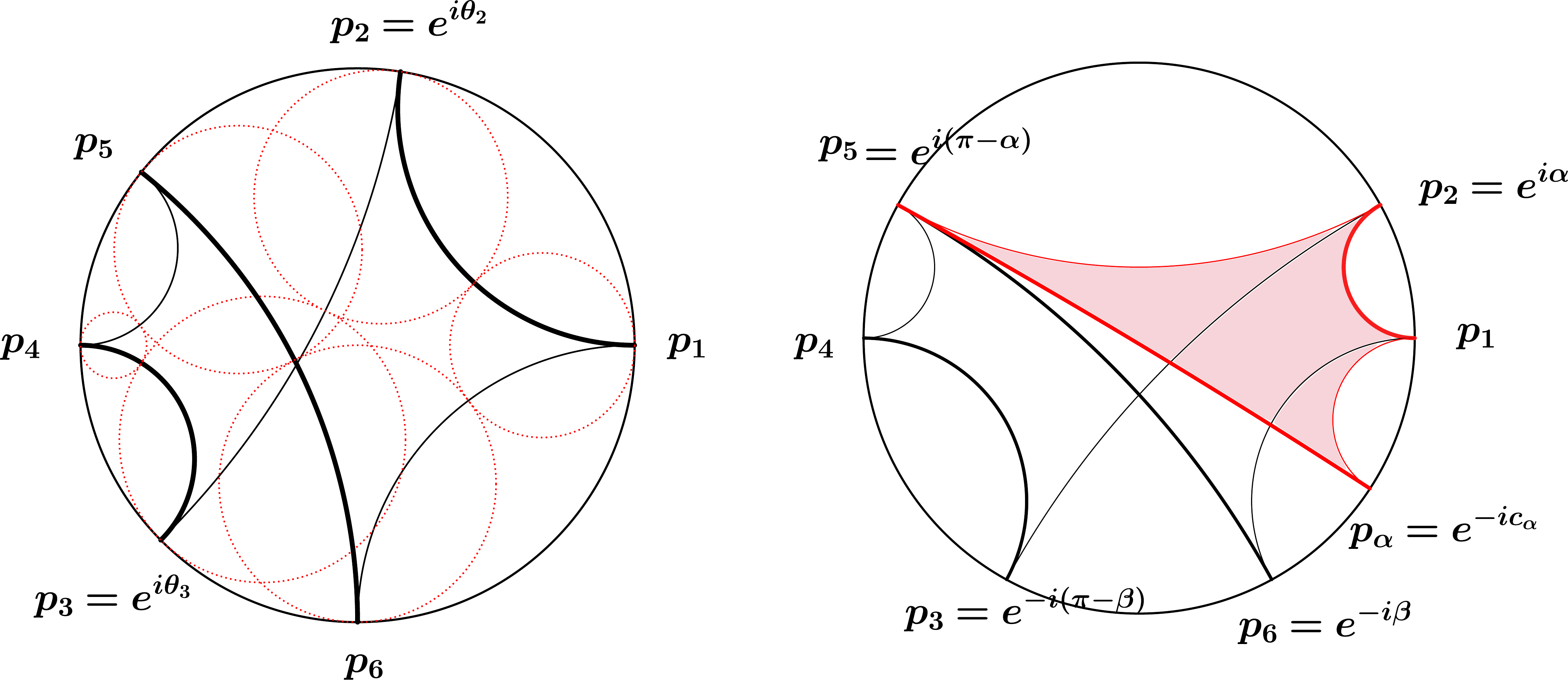}
	\end{center}
	\caption{A picture describing the parameters  of the end of a simply-connected minimal surfaces with total  curvature $-4\pi$. Left: flux condition using the horocycles to obtain the two parameters.   Right: the projection of Scherk surface (in red) acting as a barrier for the surface.}
	\label{Fig-parametros}
\end{figure}

Now let us consider a horizontal dilation along the geodesic $\{y=0\}$ (i.e. we fix $p_1,p_4$ but change $p_6$) so that $p_3$ and $p_6$ are symmetric with respect to $\{x=0\}$. Then $p_2$ and $p_5$ are also symmetric with respect to $\{x=0\}$. Therefore,
\begin{equation}\label{eq:pi}
p_1=(1,0),\ p_2=e^{i\alpha},\ p_3=e^{-i(\pi-\beta)},\ p_4=(-1,0),\ p_5=e^{i(\pi-\alpha)},\ p_6=e^{-i\beta},
\end{equation}
for some $0<\alpha\leq\beta\leq\frac\pi 2$.

There is a further restriction for $\beta$ depending on $\alpha$. Let $c_\alpha\in(0,\frac{3\pi}{2})$ be defined so that the ideal points
\[
p_1=(1,0),\ p_2=e^{i\alpha},\ p_5=e^{i(\pi-\alpha)}, p_\alpha=e^{-ic_\alpha}
\]
are the vertices of the ideal quadrilateral over  the Scherk graph is defined. We can compute $c_\alpha$ in the following constructive way: We identify $\mathbb H^2$ with $B(0,1)\subset \mathbb C$. Let $q(p_1,p_5)$ be the center of the circumference in $\mathbb C$ containing the hyperbolic geodesic $\overline{p_1p_5}$; it can be computed as the intersection between the tangent straight lines to $\partial B(0,1)$ at $p_1$ and $p_5$. As the geodesics $\overline{p_1p_5}$ and $\overline{p_2p_\alpha}$ are orthogonal; thus, the point $p_\alpha$ can be computed as the intersection of $\partial B(0,1)$ with the straight line passing by $q(p_1,p_5)$ and $p_2$. Easy computations arise  $\cos(c_\alpha)=\frac{\cos (\alpha ) (5 \cos (\alpha )-3)}{-3 \cos (\alpha )+2 \cos (2 \alpha )+3}$, see Figure~\ref{Fig-construction}. (We observe that $- 3\cos(\alpha)+ 2\cos(2\alpha) + 3 = 4\cos^2(\alpha) - 3\cos(\alpha) + 1 > 0$ .)
\begin{figure}[htb]
	\begin{center}
		\includegraphics[height=6cm]{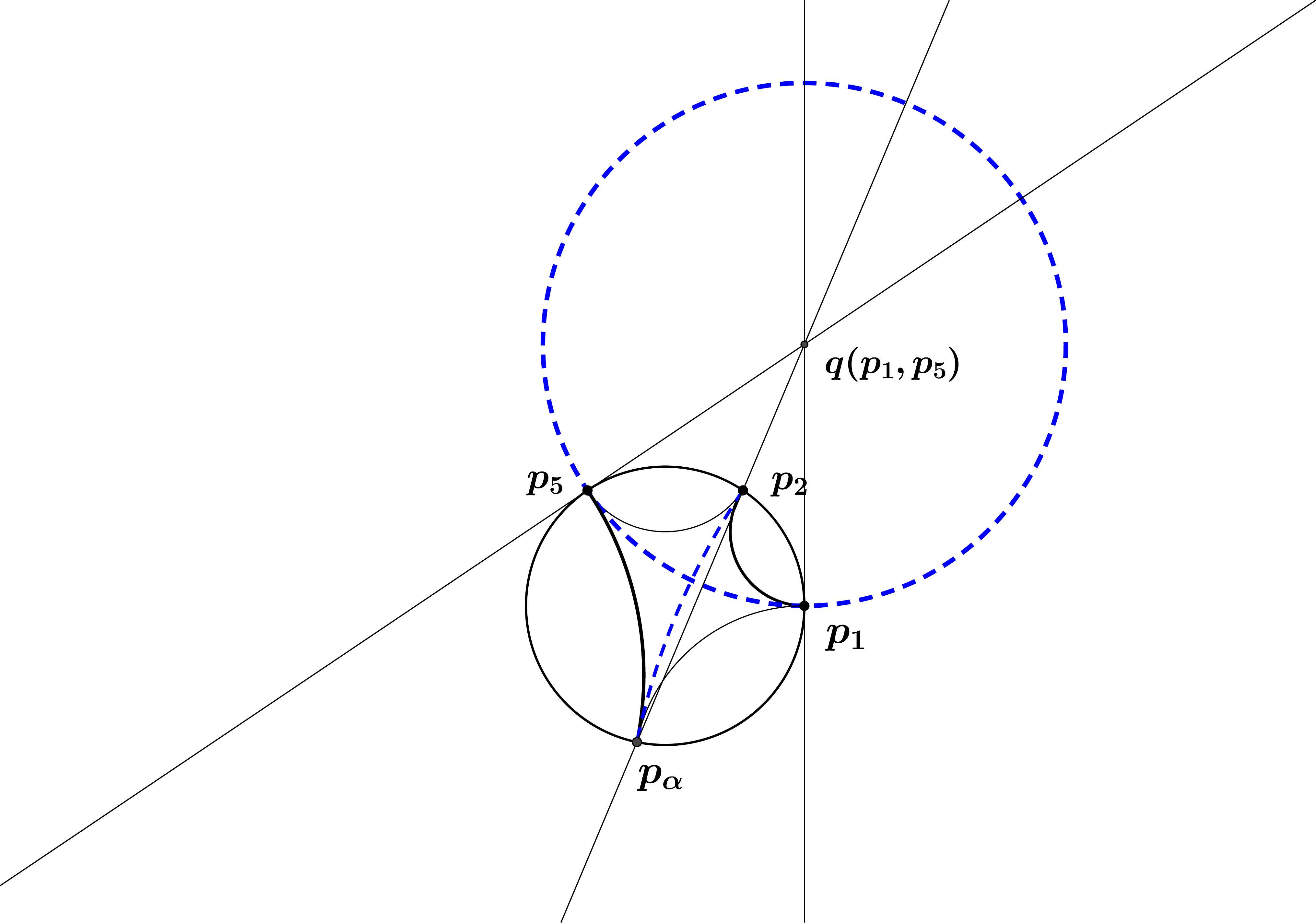}
	\end{center}
	\caption{The constructive construction of the regular quadrilateral given the points $p_1$, $p_2$ and $p_5$.}
	\label{Fig-construction}
\end{figure}

\begin{lem}
If the ideal points in~\eqref{eq:pi} determine a polygon at infinity which is the asymptotic boundary of a minimal surface with finite curvature $-4\pi$, then $\beta\leq c_\alpha$. 
\end{lem}\label{claim-parametros}
\begin{proof}
Suppose by contradiction that there exists one such minimal surface $M$ and $\beta-c_\alpha>0$.
Then there exists $\delta,\delta'\in(0,\beta-c_\alpha)$ such that there exists a Scherk graph $S$ over the quadrilateral with ideal vertices 
\[
e^{-i\delta},e^{i\alpha+\delta},\ p_5=e^{i(\pi-\alpha-\delta)}, e^{-i(c_\alpha+\delta')} .
\]
We can translate $S$ upwards so that it does not intersect $M$. Then we start translating downwards $S$. By the Maximum Principle, these translated graphs can never intersect $M$, and we reach a contradiction. Therefore, we get $\beta\leq c_\alpha$. \end{proof}

\begin{remark}
  Note that known examples exist only when $\alpha = \beta$, namely the Twisted Scherk surfaces for $k = 1$ constructed in~\cite{PR}. Whether these are unique or if one can construct further examples when $0 < \alpha < \beta < c_{\alpha}$ remains an intriguing open question. In particular, taking $\alpha$ (large enough) and $\beta = \frac{\pi}{2} < c_{\alpha}$ corresponds to Case~$(3)$ of Theorem~\ref{th:configuraciones}. Answering this question could complete the classification of embedded minimal surfaces with total curvature $-4\pi$.
\end{remark}

\subsection{Minimal surfaces with total curvature $-6\pi$ in
  $\h^2\times\R$.}
\label{sec:-6pi}

The aim of this section is to prove that an embedded minimal surface
with finite total curvature $-6\pi$ must be simply-connected and we describe its asymptotic behavior. Up to the date, the only known
examples were the corresponding Scherk graphs.

%
	


\begin{teo}\label{th:6pi}
  An embedded minimal surface with finite total curvature $-6\pi$ is
  simply-connected.
\end{teo}
\begin{proof}
	By formula~\eqref{eq} we have that $2g+2r+\sum_{i=1}^rm_i=5$,
        where $g$ is the genus, $r\geq 1$ the number of ends and $m_i+1$ is the number of geodesic in $\h^2\times\{+\infty\}$.
	\begin{itemize}
		\item If $g=1$ then $r=1$ and $m_1=1$. By
                  Corollary~\ref{cor2} the only embedded end is
                  Scherk-type. This surface cannot exist by
                  Remark~\ref{rem:scherk}, thus the genus of the surface must be zero.
		\item If $g=0$ and $r=2$ we have that $m_1=0$ and
                  $m_2=1$. Then an end is asymptotic to a vertical
                  plane and the other is of Scherk-type by
                  Corollary~\ref{cor2}. Since the surface is embedded,
                  these ends do not intersect each other. This surface
                  cannot exist by Proposition~\ref{th:Alexandrov1}.
	\end{itemize}
	
	Then the only possibility is $g=0$, $r=1$ and $m_1=3$; in
        particular $M$ is simply-connected. 	
\end{proof}
\begin{teo}\label{th:configuraciones-2}
  Let $M\subset \h^2\times \r$ be a complete embedded minimal surface
  with finite total curvature $-6\pi$ and asymptotic polygon $\mathcal P$. Then $\mathcal P$ lies in one of the following cases:
  \begin{itemize}
  \item   If $\mathcal P$ is generic, it must be in one of the following four cases (up to an isometry and a possible reordering of the vertices of $\pi(\mathcal P)$):
  \begin{enumerate}
  \item $0=\theta_1<\theta_2<...<\theta_8<2\pi$ (the ideal points
    $p_i$ are cyclically ordered).  In this case, $M$ must be a Scherk graph over an ideal geodesic octagon.
  \item    $0=\theta_1<\theta_2<\theta_3<\theta_6<\theta_5<\theta_4<\theta_7<\theta_8<2\pi$.
\item $0=\theta_1<\theta_2<\theta_7<\theta_4<\theta_5<\theta_6<\theta_3<\theta_8<2\pi$.
  \item    $0=\theta_1<\theta_2<\theta_3<\theta_6<\theta_7<\theta_8<\theta_5<\theta_4 <2\pi$.
  \end{enumerate}
  \item If $\mathcal P$ is not generic must be in one the sub-cases as a limit of a generic polygon at infinity  (see Figure~\ref{Fig-configuraciones-2}):
  \begin{enumerate}
  \item[(2.a)] $0=\theta_1<\theta_2<\theta_3=\theta_6<\theta_5<\theta_4<\theta_7<\theta_8<2\pi$.  
  \item[(2.b)] $0=\theta_1<\theta_2<\theta_3=\theta_6<\theta_5<\theta_4=\theta_7<\theta_8<2\pi$.
  \item[(3.a)] $0=\theta_1<\theta_2<\theta_7=\theta_4<\theta_5<\theta_6<\theta_3<\theta_8<2\pi$.  
  \item[(3.b)] $0=\theta_1<\theta_2<\theta_7=\theta_4<\theta_5<\theta_6=\theta_3<\theta_8<2\pi$.
  \item[(3.c)] $0=\theta_1<\theta_2=\theta_7=\theta_4<\theta_5<\theta_6=\theta_3<\theta_8<2\pi$.  
  \item[(3.d)] $0=\theta_1<\theta_2=\theta_7=\theta_4<\theta_5<\theta_6=\theta_3=\theta_8<2\pi$.  
  \item[(3.e)] $0=\theta_1<\theta_2=\theta_7<\theta_4<\theta_5<\theta_6<\theta_3=\theta_8<2\pi$.  
  \item[(4.a)] $0=\theta_1<\theta_2<\theta_3=\theta_6<\theta_7<\theta_8<\theta_5<\theta_4 <2\pi$. 
  \item[(4.b)] $0=\theta_1<\theta_2<\theta_3=\theta_6<\theta_7<\theta_8=\theta_5<\theta_4 <2\pi$. 
  \item[(4.c)] $0=\theta_1<\theta_2<\theta_3=\theta_6<\theta_7<\theta_8=\theta_5<\theta_4 =2\pi$. 
  \item[(4.d)] $0=\theta_1<\theta_2<\theta_3<\theta_6<\theta_7<\theta_8=\theta_5<\theta_4 =2\pi$. 
  \item[(4.e)]  $0=\theta_1<\theta_2<\theta_3<\theta_6<\theta_7<\theta_8=\theta_5<\theta_4 <2\pi$. 
      \end{enumerate}
  \end{itemize}
  \end{teo}

\begin{figure}[htb]
	\begin{center}
		\includegraphics[height=12cm]{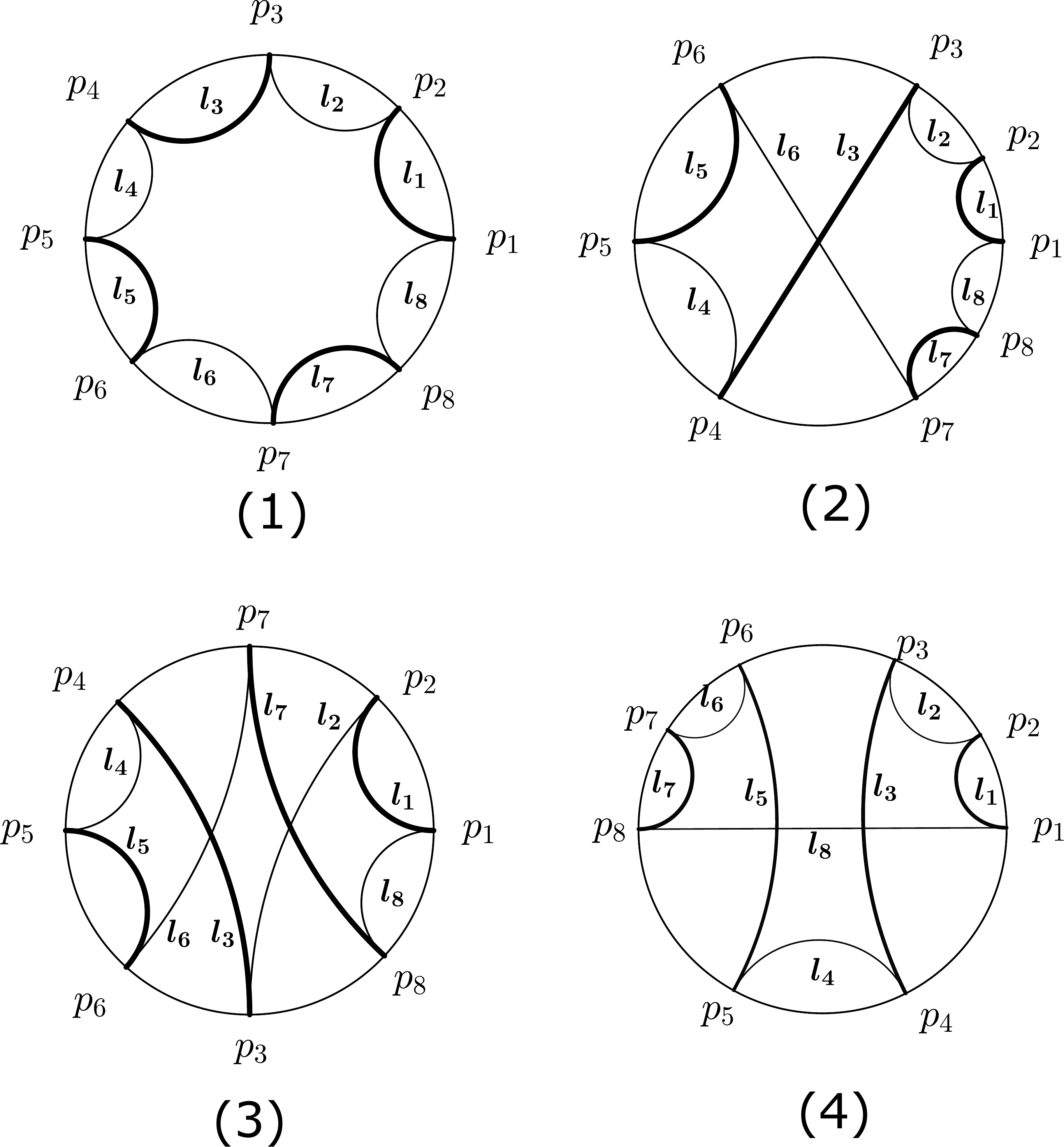}
	\end{center}
	\caption{Projections of the possibles configurations of the generic polygons at infinity in Theorem~\ref{th:configuraciones-2}. }
	\label{Fig-configuraciones-2}
\end{figure}
\begin{proof}
  Let $M$ be asymptotic to $\mathcal{P}$ and use the notation introduced above. By Proposition~\ref{prop:limits-polygon}, we know that $\mathcal{P}$ is either generic or the limit of generic polygons at infinity. We then first study the possible cases for $\mathcal{P}$ when it is generic. We analyze the number of (transverse) intersections between geodesics with subscripts of different parity as by Lemma~\ref{l1} no geodesics with the same parity can intersect.

If there is no intersection, we arrive at case \emph{(1)}.

If there is exactly one intersection, then, after renaming the geodesics if necessary, we can assume that \( l_3 \) and \( l_6 \) intersect. In this case, the only possibility corresponds to case \emph{(2)}.

If there are two intersections, we have two possibilities:
\begin{itemize}
    \item Four geodesics intersect in pairs. Again, after renaming the geodesics if necessary, they must be \( l_3 \) with \( l_6 \) and \( l_2 \) with \( l_7 \), leading to case \emph{(3)}.
    \item One geodesic intersects two different geodesics. After renaming the geodesics if necessary, we assume that \( l_8 \) intersects both \( l_3 \) and \( l_5 \), leading to case \emph{(4)}.
\end{itemize} 
The case of more than two intersections is not possible.

To complete the proof, it suffices to consider the possible limit configurations resulting from the confluence of consecutive vertices in $\mathbb{S}^1$ that are not the endpoints of the same geodesic.
\end{proof}

\begin{remark}
Although Theorems~\ref{th:configuraciones} and ~\ref{th:configuraciones-2} are stated for complete embedded minimal surfaces, we could obtain similar results for minimal  embedded ends of finite total curvature.      
\end{remark}

\section{New examples of surfaces with  total curvature $-6\pi$}\label{sec:examples}
 Observe that   case \emph{(1)} in Theorem~\ref{th:configuraciones-2}  is completely classified as the complete Scherk-graphs (see~\cite[Theorem 7]{HMR} and  Proposition~\ref{th:Alexandrov1}).
 This section is devoted to the construction of some new minimal disks in $\mathbb H^2\times \R $ with some of the asymptotic behavior described by the configurations \emph{(2)} and \emph{(3)}  of Theorem~\ref{th:configuraciones-2}, see Figure~\ref{Fig-Superficies}. Even if the ends of these examples are embedded, we do not know whether these surfaces are embedded (although we expect that they are.)
 
 In~\cite{Coskunuzer}, Coskunuzer studied the asymptotic Plateau problem in $\mathbb H^2\times\R$ for both area-minimizing and minimal surfaces. He gave necessary and sufficient conditions under certain extra hypothesis. 
 The examples we provide can be seen as solutions to some asymptotic Plateau problem for minimal surfaces in the exceptional case $\alpha(\mathcal P)=\beta(\mathcal P)$ (using the notation in Coskunuzer's paper), not treated in general in~\cite{Coskunuzer}. 
 
 \begin{figure}[htb]
	\begin{center}
		\includegraphics[height=8cm]{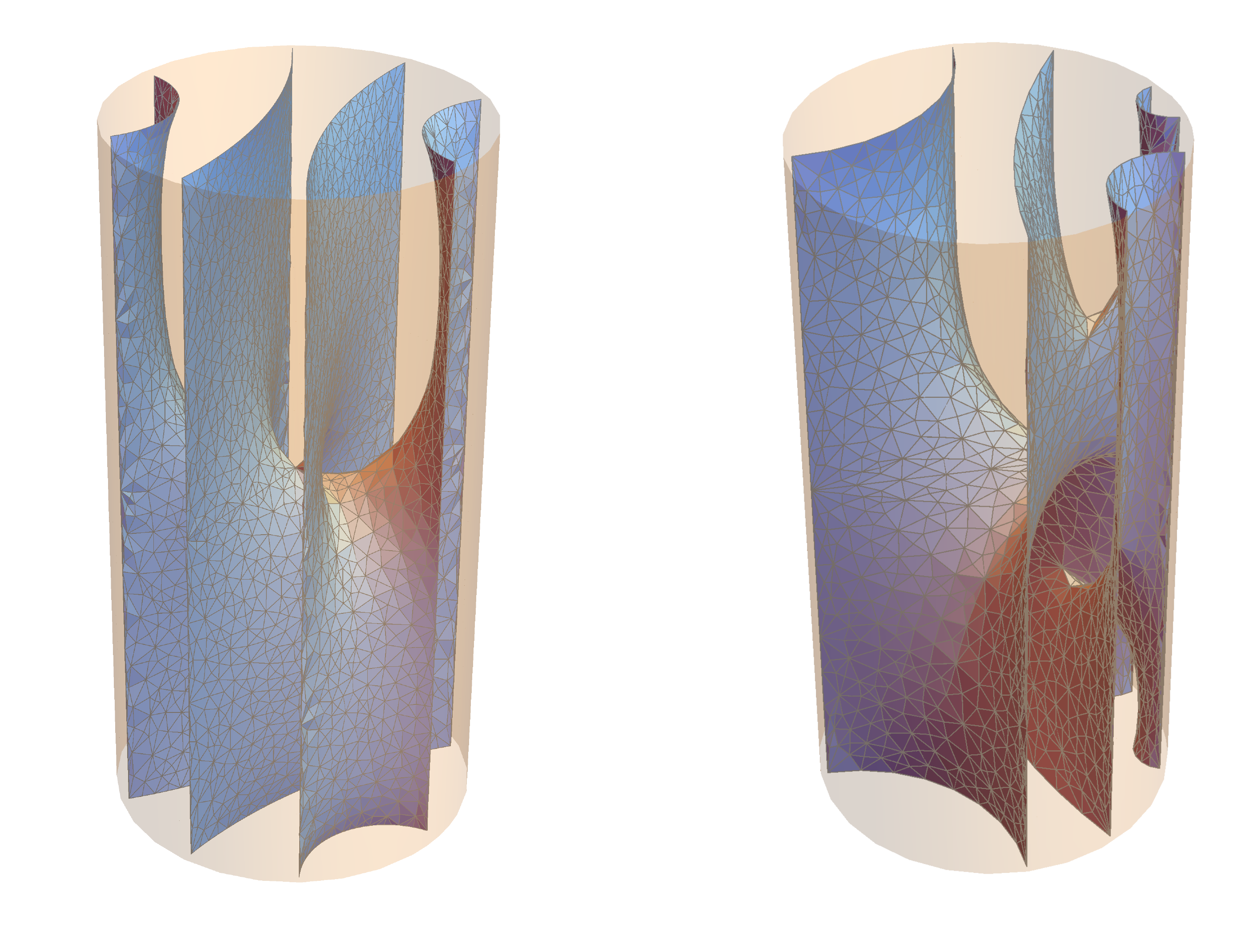}
	\end{center}
	\caption{The surfaces constructed in Proposition~\ref{prop: Ejemplo 1} and~\ref{prop: Ejemplo 2} that realizes cases \it{(3)} (left) and \it{(2)} (right) of Theorem~\ref{th:configuraciones-2}.} 
	\label{Fig-Superficies}
\end{figure}

\begin{prop}\label{prop: Ejemplo 1}
    There exists a one-parameter family of simply-connected minimal surfaces with total curvature $-6\pi$ with embedded end asymptotic to a polygon at infinity lying in case (3) of Theorem~\ref{th:configuraciones-2}. 
\end{prop}

\begin{proof}
We will show that any surface in this family can be obtained by successive Schwarz reflections from an area-minimizing surface, obtained as the limit of a sequence of area-minimizing surfaces solving suitable Plateau problems.

Recall we are denoting by $\overline{\mathbb H^2\times\mathbb R}$ the product compactification of $\mathbb H^2\times\mathbb R$. For every $r,h\in(0,\infty]$, $\overline{\mathbb H^2\times\mathbb R}$ is diffeomorphic to the compact cylinder $\overline B(r)\times[-h,h]$, where $\overline B(r)\subset\mathbb H^2$ is the closed geodesic ball of (hyperbolic) radius $r$ centered at the origin; hence we can identify them. 
Observe that when $r=h=\infty$, this identification is simply the identity. Although $r$ is the hyperbolic distance to the origin in $\mathbb H^2$ we will continue identifying  $\partial_\infty \mathbb H^2$ with $\mathbb S^1\subset\mathbb C$ for the notation of ideal points.

Consider in $\overline{\mathbb H^2\times\mathbb R}$ the admissible polygon at infinity $\mathcal P_\alpha$ with vertices (see Figure~\ref{Fig-Poligono-1})
\begin{align*}
   &q_1=(1,+\infty),\ 
q_2=(e^{i\alpha},+\infty),\ 
q_3=(e^{i\alpha},-\infty),\\
&q_4=(-i,-\infty),\
q_5=(-i,0),\
q_6=(0,0),\
q_7=(1,0),
\end{align*}
for $0<\alpha\le\frac{\pi}{2}$. 

\begin{figure}[htb]
	\begin{center}
		\includegraphics[height=6cm]{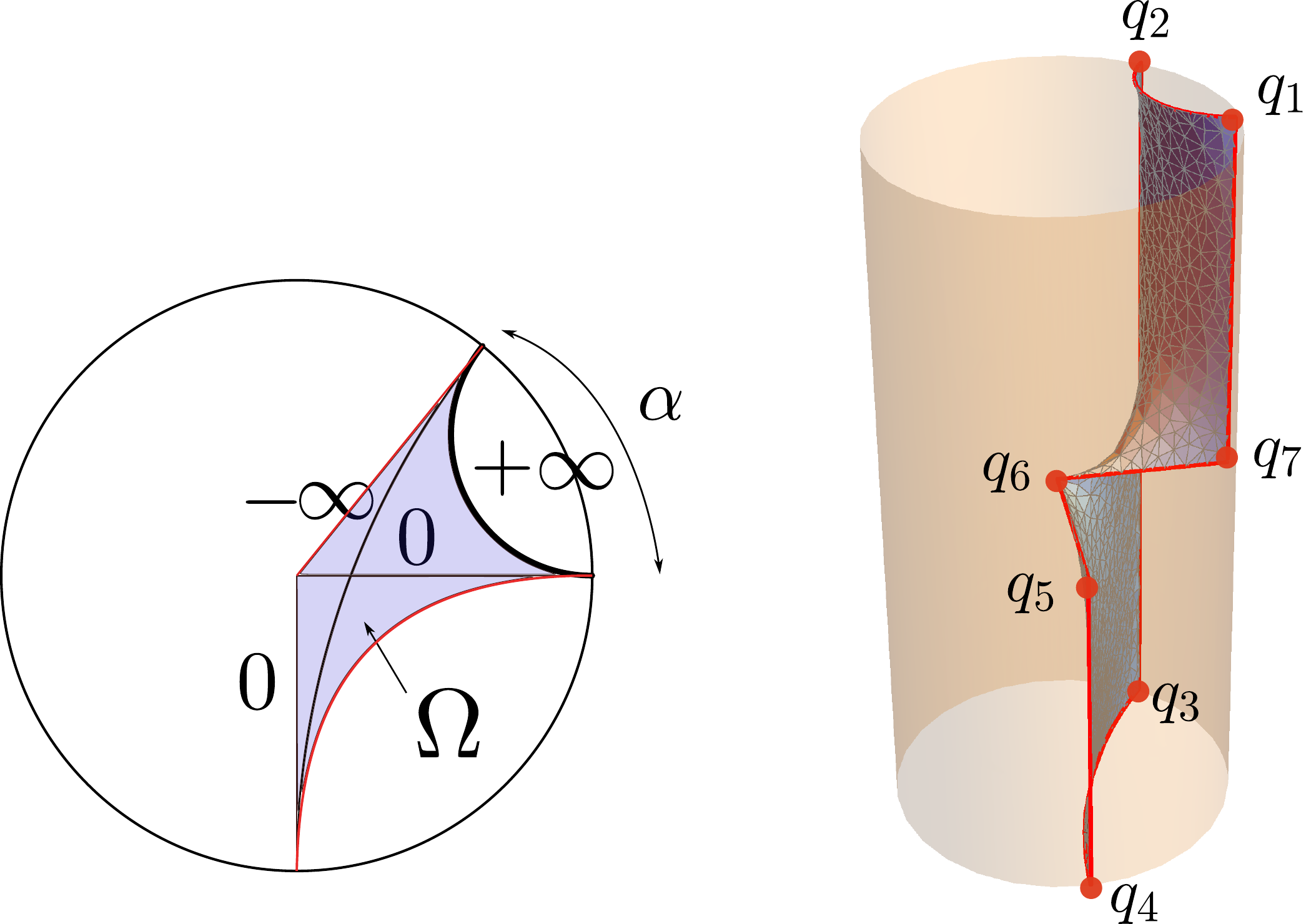}
	\end{center}
	\caption{Left: Vertical projection of the admissible polygon at infinity $\mathcal  P_{\alpha}$. Right: The admissible polygon at infinity $\mathcal P_\alpha$ (in red) and the area-minimizing surface $\Sigma_\alpha$ spanned by it. }
	\label{Fig-Poligono-1}
\end{figure}

Denote by $\overline{q_iq_j}$ the (horizontal or vertical) geodesic joining $q_i$ and $q_j$. Then the polygon $\mathcal P_\alpha$ consists of the horizontal ideal geodesics
\[
\overline{l}_1=\overline{q_1q_2}\subset\mathbb H^2\times\{+\infty\},
\qquad
\overline{l}_2=\overline{q_3q_4}\subset\mathbb H^2\times\{-\infty\},
\]
the horizontal geodesics
\[
h_1=\overline{q_5q_6}\subset\mathbb H^2\times\{0\},
\qquad
h_2=\overline{q_6q_7}\subset\mathbb H^2\times\{0\},
\]
and the vertical ideal lines
\[
v_1=\overline{q_7q_1},\qquad
v_2=\overline{q_2q_3},\qquad
v_3=\overline{q_4q_5}.
\]

We denote by $\mathcal P_\alpha^{r,h}$ the image of $\mathcal P_\alpha$ under the above identification of $\overline{\mathbb H^2\times\mathbb R}$ with $B(r)\times[-h,h]$. Since the cylinder $B(r)\times[-h,h]$ is mean convex, Meeks--Yau's Theorem~\cite{M-Y} guarantees the existence of an (embedded) area-minimizing surface $\Sigma_\alpha^{r,h}$ with boundary $\partial\Sigma_\alpha^{r,h}=\mathcal P_\alpha^{r,h}$ for every $r,h>0$.

We are going to construct the area-minimizing surface $\Sigma_\alpha$ spanned by $\mathcal P_\alpha$ by considering two successive limiting processes from $\Sigma_\alpha^{r,h}$. We first fix the height $h$ and let the radius $r$ tend to infinity. Next we let $h$ tend to infinity.

Fix $h>0$ and let $r\to\infty$. Clearly, $\mathcal P_\alpha^{r,h}$ converges to 
$\mathcal P_\alpha^{\infty,h}=\mathcal P_\alpha^{h}$ in $\overline{\mathbb H^2\times\mathbb R}$, 
where $\mathcal P_\alpha^{h}$ is the image of $\mathcal P_\alpha$ under the identification of $\overline{\mathbb H^2\times\mathbb R}$ with $\overline{\mathbb H^2}\times[-h,h]$.

By the maximum principle, each surface $\Sigma_\alpha^{r,h}$ is contained in the horizontal slab between the horizontal planes $\mathbb H^2\times\{-h\}$ and $\mathbb H^2\times\{h\}$, intersecting them only along its boundary. On the other hand, let us denote by $\Omega$ the convex hull of $\pi(\mathcal P_\alpha^{h})$. Another application of the maximum principle, this time involving vertical planes, shows that $\Sigma_\alpha^{r,h}\subset\Omega\times(-h,h)$, for any $r>0$. 

We observe that $h_1\subset\partial\Omega$ (where we are identifying $\mathbb H^2$ with $\mathbb H^2\times\{0\}$). Hence we can extend $\Sigma_\alpha^{r,h}$ by reflection symmetry about $h_1$ obtaining an embedded minimal surface $\widetilde \Sigma_\alpha^{r,h}$. 
Given any fixed $r_0>0$, for any $r>r_0$ the surface $\widetilde \Sigma_\alpha^{r,h}$ contains the horizontal geodesic segment $h_1^{r_0}=h_1\cap \mathcal P_\alpha^{r_0,h}$. Hence the sequence of surfaces $\widetilde\Sigma_\alpha^{r,h}$ (and then $\Sigma_\alpha^{r,h}$ as well) has accumulations points and cannot escape to infinity as $r\to\infty$.

Since $\mathbb H^2\times\mathbb R$ has bounded geometry and every $\Sigma_\alpha^{r,h}$ is area-minimizing, hence stable, then the standard compactness theorem for stable embedded minimal surfaces says that $\Sigma_\alpha^{r,h}$ converges smoothly on compact subsets with the $C^k$ topology for every $k\ge0$  (after passing to a subsequence) to an embedded minimal surface $\Sigma_\alpha^{h}\subset\Omega\times[-h,h]$. We observe that $\Sigma_\alpha^{h}$ is area-minimizing as it is the limit of area-minimizing surfaces. 

Since for any fixed $r_0>0$ we get that $\partial\Sigma_\alpha^{r,h}$ contains $h_1^{r_0}$, for any $r>r_0$, we deduce $h_1^{r_0}\subset\partial\Sigma_\alpha^{h}$. Taking $r_0\to+\infty$ we deduce that $h_1\subset\partial\Sigma_\alpha^{h}$. A similar argument shows that $h_2\subset\partial\Sigma_\alpha^{h}$.

Let us call $\Gamma_\alpha^{h}=\partial\Sigma_\alpha^{h}\cup\partial_\infty\Sigma_\alpha^{h}$ and prove that $\Gamma_\alpha^{h}=\mathcal P_\alpha^{h}$. Since $\mathcal P_\alpha^{r,h}$ converges to $\mathcal P_\alpha^{h}$ as $r\to+\infty$, every point of $\mathcal P_\alpha^{h}$ belongs to the boundary of the limit surface $\Sigma_\alpha^{h}$. Thus $\mathcal P_\alpha^{h}\subset\Gamma_\alpha^{h}$.
On the other hand, since $\Sigma_\alpha^{h}\subset\Omega\times[-h,h]$ and every point $p\in\partial_\infty({\mathbb H^2\times\mathbb R})\setminus\mathcal P_\alpha^{h}$ lies  outside $\partial_\infty\Omega\times[-h,h]$, hence $p\notin\Gamma_\alpha^{h}$. We can then deduce $\Gamma_\alpha^{h}=\mathcal P_\alpha^{h}$.

We now let $h\to\infty$. Let $S^+$ be the hyperbolic-invariant minimal graph taking the value $+\infty$ over the projection of $\overline{l}_1$ and the value $0$ over the arc $\gamma\subset\partial_\infty\mathbb H^2$ with endpoints $\pi(q_1)=1$ and $\pi(q_2)=e^{i\alpha}$ that contains $-1$ (see Figure~\ref{Fig-barreras-1}). For any small $\delta>0$, define $S^+_\delta$ as the horizontal hyperbolic translation of $S^+$ taking the value $+\infty$ over the geodesic with endpoints $e^{-i\delta}$ and $e^{i(\alpha+\delta)}$. For $r$ large enough we prove by the maximum principle with vertical translations of $S^+_\delta$ that $S^+_\delta$ is an upper barrier for the surfaces $\Sigma_\alpha^{r,h}$. Since this holds for any $r,\delta>0$,  we obtain that $S^+$ is an upper barrier for $\Sigma_\alpha^{h}$ as well, for any $h>0$.

Let us now consider $S^-_1$ and $S^-_2$ the hyperbolic-invariant minimal graphs taking the value $-\infty$ over the projection of $\overline{l}_2$ and each of them taking the value $0$ in a different component of $\partial_\infty\mathbb H^2\setminus\{e^{i\alpha},-i\}$ (see Figure~\ref{Fig-barreras-1}). Similarly as above we prove that $S^-_1\cup S^-_2$ is a lower barrier for $\Sigma_\alpha^{h}$, for any $h>0$. 

Therefore the stable minimal surface $\Sigma_\alpha^{h}$ is contained in the region $\mathcal{R}$ of $\Omega\times\R$ below $S^+$ and above $S^-_1\cup S^-_2$, for any $h>0$. Moreover, 
 $h_1\cup h_2\subset \Sigma_\alpha^{h}$ for any $h$, so the sequence has accumulation points. Therefore, arguing as above, after passing to a subsequence, $\Sigma_\alpha^{h}$ converges smoothly on compact subsets, with the $C^k$ topology for every $k\ge0$, to an embedded minimal surface $\Sigma_\alpha\subset\mathcal{R}$. 
 Since $\mathcal P_\alpha^{h}$ converges to $\mathcal P_\alpha$, we conclude  $\mathcal P_\alpha\subset \partial\Sigma_\alpha\cup\partial_\infty\Sigma_\alpha$. Moreover, $\partial_\infty\mathcal{R}=\mathcal P_\alpha\cap \partial_\infty(\h^2\times\R)$, from where we deduce $\partial\Sigma_\alpha\cup\partial_\infty\Sigma_\alpha=\mathcal P_\alpha$.
 

\begin{figure}[htb]
	\begin{center}
		\includegraphics[height=4cm]{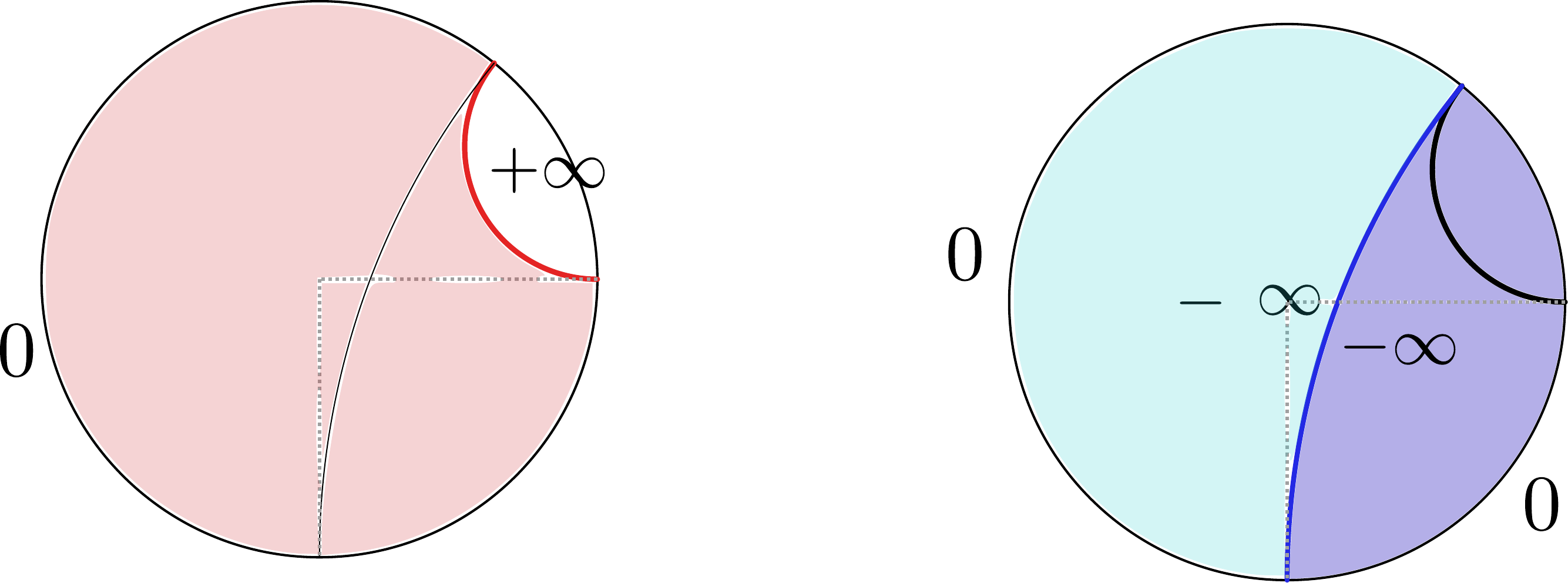}
	\end{center}
	\caption{The projections of the upper (left) and lower (right) barriers used in Proposition \ref{prop: Ejemplo 2}.}
	\label{Fig-barreras-1}
\end{figure}

Since the area-minimizing property is preserved under smooth convergence on compact subsets, $\Sigma_\alpha$ is itself an area-minimizing surface. In particular, it is a smoothly embedded minimal surface with boundary $\mathcal P_\alpha$. Reflecting $\Sigma_\alpha$ across the horizontal geodesics $h_1$ and $h_2$ by Schwarz reflection, we obtain a complete proper minimal surface $\bar\Sigma_\alpha$ asymptotic to an embedded admissible polygon at infinity of type \emph{(3)} for $0<\alpha<\frac{\pi}{2}$, and to the limit case \emph{(d)} when $\alpha=\frac{\pi}{2}$, in the sense of Theorem~\ref{th:configuraciones-2}. We observe that the end of $\bar\Sigma_\alpha$ is embedded by item \emph{(5)} of Theorem~\ref{th:ctf}. Finally, $\bar\Sigma_\alpha$ has total curvature $-6\pi$ by Theorems~\ref{th:caracterizacion-ctf} and item \emph{(6)} in~\ref{th:ctf}.
\end{proof}

\begin{remark}
In the previous theorem, if we reflect $\Sigma_\alpha$ about $h_1$ we obtain an embedded surface. However, the global embeddedness of $\bar\Sigma_\alpha$ may fail after the Schwarz reflection about $h_2$. 
\end{remark}

\begin{remark}[Examples with higher total curvature]\label{remark: Ejemplo k}
Using the same ideas as in Proposition~\ref{prop: Ejemplo 1}, for each $k\geq2$ one can similarly construct a one-parameter family of minimal surfaces $\bar \Sigma_\alpha^k$ with total curvature $-2(2k-1)\pi$, where $0<\alpha\leq\frac{\pi}{k}$. The construction is analogous, except that we take $\pi(q_4)=\pi(q_5)=e^{-i\pi/k}$, which means that the two horizontal geodesics $h_1$ and $h_2$ meet at an angle of $\frac{\pi}{k}$. Note that the complete surface $\bar\Sigma_\alpha^k$ has $k$ complete horizontal geodesics in $\mathbb H^2\times\{0\}$ meeting at an angle of $\frac{\pi}{k}$; see Figure~\ref{Fig-Varios-LADOS}. 
\end{remark}

\begin{figure}[htb]
	\begin{center}
		\includegraphics[height=8cm]{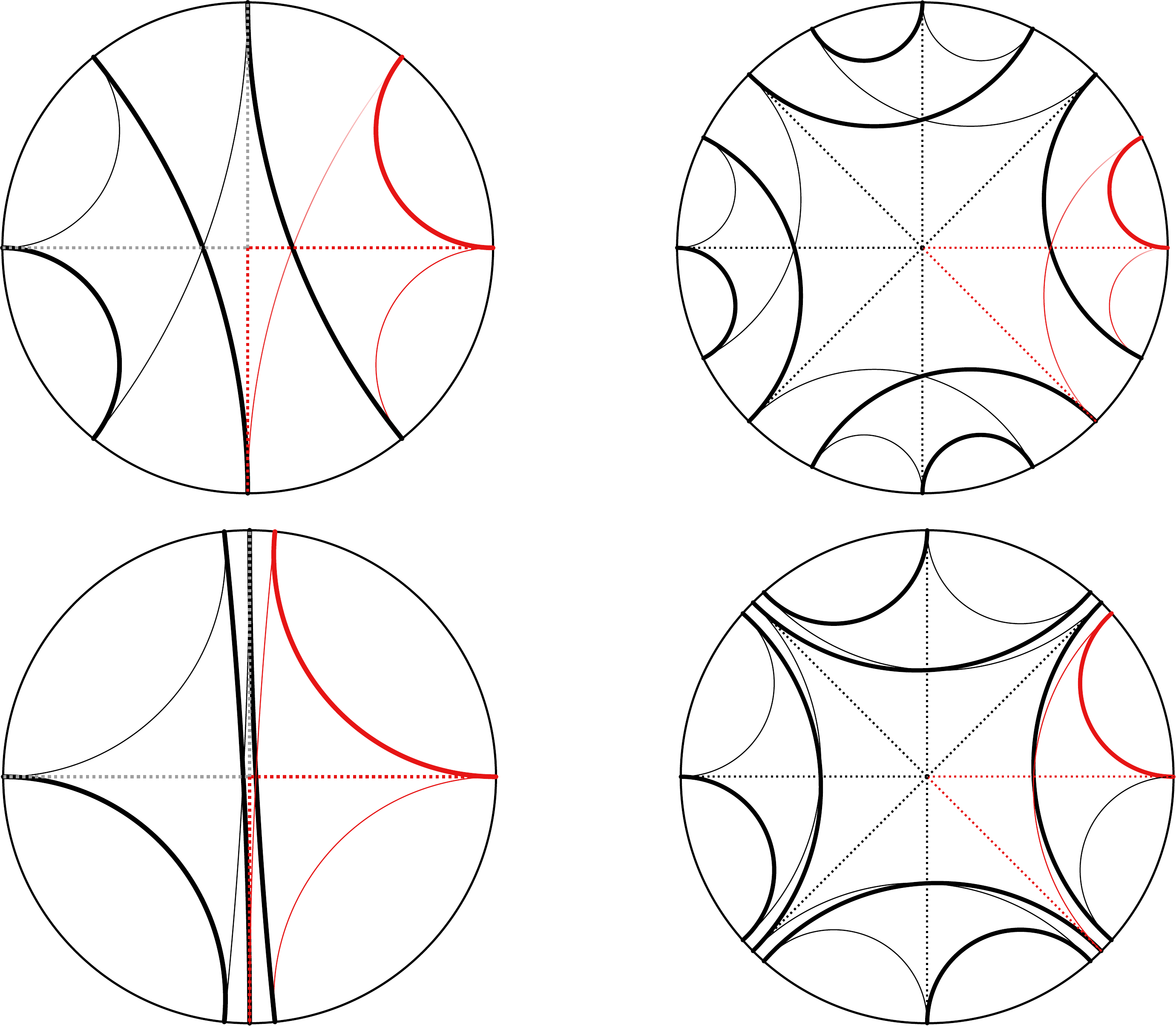}
	\end{center}
	\caption{Left: Projection of the  surface $\bar\Sigma_ {\alpha}^2$. Right: Projection of the $\bar\Sigma_ {\alpha}^4$.\\ The figures below show cases in which $\alpha$ is close to $\frac\pi k$.  }
	\label{Fig-Varios-LADOS}
\end{figure}

\begin{prop}\label{prop: Ejemplo 2}
    There exists a two-parameter family of simply-connected minimal surfaces with total curvature $-6\pi$ with embedded end asymptotic to a polygon at infinity lying in case (2) of Theorem~\ref{th:configuraciones-2}. 
\end{prop}
\begin{proof}
We follow the same strategy as in Proposition~\ref{prop: Ejemplo 1}, so we only sketch the main arguments.

Assume that $0<\beta<\alpha\leq\frac{\pi}{2}$ and consider the following points in $\overline{\mathbb H^2\times\mathbb R}$  (see Figure~\ref{Fig-Poligono-2}):
\begin{align*}
     &q_1=(1,-\infty),\ q_2=(1,0),\ q_3=(-1,0),\ q_4=(-1,+\infty),\\
     &q_5=(i,+\infty),\ q_6=(i,-\infty),\ q_7=(e^{-i\alpha},-\infty),\ q_8=(e^{-i\alpha},+\infty),\\
     &q_9=(e^{-i\beta},+\infty),\ q_{10}=(e^{-i\beta},-\infty).
\end{align*}

\begin{figure}[htb]
	\begin{center}
		\includegraphics[height=7cm]{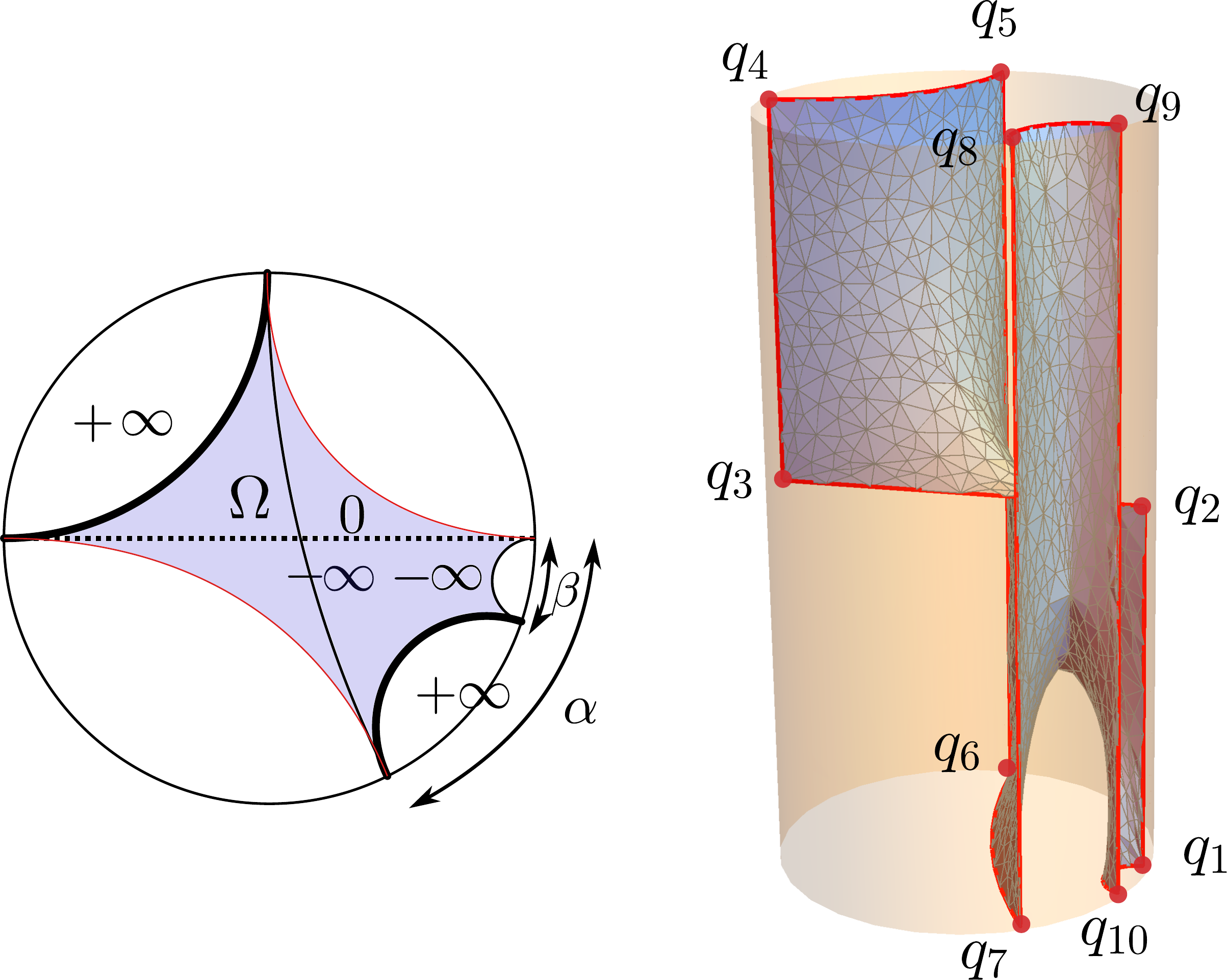}
	\end{center}
	\caption{Left: The projection of the admissible polygon at infinity.  Right: The admissible polygon at infinity $\mathcal P_{\alpha,\beta}$ (in red) and the area-minimizing surface $\Sigma_{\alpha,\beta}$ spanned by it.}
	\label{Fig-Poligono-2}
\end{figure}

Denote by $\overline{q_iq_j}$ the (horizontal or vertical) geodesic joining $q_i$ and $q_j$. We consider the polygon $\mathcal P_{\alpha,\beta}$ formed by the following horizontal and vertical geodesic segments:
\begin{align*}
     &v_1=\overline{q_1q_2},\ h_1=\overline{q_2q_3}\subset\h^2\times\{0\}\\
     &v_2=\overline{q_3q_4},\ \bar l_1=\overline{q_4q_5}\subset\h^2\times\{+\infty\},\\
     &v_3=\overline{q_5q_6},\ \bar l_2=\overline{q_6q_7}\subset\h^2\times\{-\infty\},\\
     &v_4=\overline{q_7q_8},\ \bar l_3=\overline{q_8q_9}\subset\h^2\times\{+\infty\},\\
     &v_5=\overline{q_9q_{10}},\ \bar l_4=\overline{q_{10}q_1}\subset\h^2\times\{-\infty\}.
\end{align*}

Let $\Omega_1\subset\mathbb H^2$ be the domain bounded by the geodesics $\pi(\bar l_1)$ and $\pi(\bar l_3)$ together with the two arcs $C_1,C_2\subset\partial_\infty\mathbb H^2$ joining their endpoints. Let $\Omega_2\subset\mathbb H^2$ be the domain bounded by the geodesics $\pi(\bar l_2)$ and $\pi(\bar l_4)$ together with the two arcs $C_3,C_4\subset\partial_\infty\mathbb H^2$ joining their endpoints. Finally, let $\Omega_3\subset\mathbb H^2$ be the domain bounded by $\pi(h_3)$ and the arc $C_5\subset\partial_\infty\mathbb H^2\setminus\pi(h_3)$ containing the point $-1\in\partial_\infty\mathbb H^2$.

Let's consider the  following Jenkins--Serrin problems (see~\cite{MRR}):
\[
(I)\;
\begin{cases}
\operatorname{div}\!\left(\dfrac{\nabla u}{\sqrt{1+|\nabla u|^2}}\right)=0,& u\in\Omega_1,\\[0.2cm]
u=+\infty,  &\hspace{-1.4cm}\text{on }\pi(\bar{l_1})\cup\pi(\bar{l_3}),\\
u=0,& \hspace{-1.4cm}   \text{on }C_1\cup C_2,
\end{cases}
\qquad
(II)\;
\begin{cases}
\operatorname{div}\!\left(\dfrac{\nabla u}{\sqrt{1+|\nabla u|^2}}\right)=0,& u\in\Omega_2,\\[0.2cm]
u=-\infty,& \hspace{-1.4cm} \text{on }\pi(\bar{l_2})\cup\pi(\bar{l_4}),\\
u=0,& \hspace{-1.4cm} \text{on }C_3\cup C_4.
\end{cases}
\]

It is easy to verify that for this choice of points $q_i$, the Jenkins-Serrin conditions (see~\cite{CR}) are satisfied for $\Omega_1$. (In this case, these conditions are equivalent to the fact that $\Omega_1$ contains an ideal regular quadrilateral in its interior.) Hence the problem $(I)$  admits a solution for every $0<\beta<\alpha\leq\frac{\pi}{2}$. 

On the other hand, for each $\alpha\in(0,\frac{\pi}{2}]$ there exists a unique value $\beta_0(\alpha)$ such that the ideal quadrilateral with vertices $i$, $e^{-i\alpha}$, $e^{-i\beta_0(\alpha)}$ and $1$ is regular. A straightforward computation, analogous to that in Lemma~\ref{claim-parametros}, shows that
\[
\cos\bigl(\beta_0(\alpha)\bigr)=
\frac{2\bigl(\sin\alpha+\cos\alpha+1\bigr)}
{2\sin\alpha+\cos\alpha+3}.
\]
Therefore, the Jenkins--Serrin conditions~\cite{CR} for the problem $(II)$ are satisfied whenever the regular quadrilateral is contained in $\Omega_2$, or equivalently
$0<\alpha\leq\frac{\pi}{2}$ and $0<\beta<\beta_0(\alpha)$.

Let $S^+$ denote the solution to problem $(I)$, and let
$S^-=S_1^-\cup S_2^-$, where $S_1^-$ is the solution to problem $(II)$ and $S_2^-$ is the hyperbolic-invariant minimal graph defined on $\Omega_3$ taking values $-\infty$ on $\pi(\bar{l_2})$ and $-1$ on $C_5$.

As in Proposition~\ref{prop: Ejemplo 1}, let $\mathcal P_{\alpha,\beta}^{r,h}$ denote the image of $\mathcal P_{\alpha,\beta}$ under the natural identification of $\overline{\mathbb H^2\times\mathbb R}$ with $B(r)\times[-h,h]$. Since the cylinder $C(r,h)$ is mean convex, Meeks--Yau's Theorem~\cite{M-Y} yields an embedded area-minimizing surface $\Sigma_{\alpha,\beta}^{r,h}$ with boundary $\mathcal P_{\alpha,\beta}^{r,h}$ for every $r,h>0$.

Let $r_n,h_n\to\infty$. Arguing as in Proposition~\ref{prop: Ejemplo 1}, the barriers $S^+$ and $S^-$, together with the maximum principle, imply that the sequence $\Sigma_{\alpha,\beta}^{r_n,h_n}$ admits a subsequence converging smoothly on compact subsets, in the $C^k$ topology for every $k\ge0$, to a connected embedded area-minimizing surface $\Sigma_{\alpha,\beta}$ with boundary $\mathcal P_{\alpha,\beta}$ (see Figure~\ref{Fig-barreras-2}).

Finally, reflecting $\Sigma_{\alpha,\beta}$ across the horizontal geodesic $h_1$ by Schwarz reflection, we obtain a complete proper minimal surface $\bar\Sigma_{\alpha,\beta}$ asymptotic to an embedded admissible polygon at infinity of type \emph{(2)}, and to the limit configuration \emph{(b)} in Theorem~\ref{th:configuraciones-2} when $\alpha=\frac{\pi}{2}$. Observe that global embeddedness may fail after the Schwarz reflection. In particular, Theorems~\ref{th:caracterizacion-ctf} and~\ref{th:ctf} imply that $\bar\Sigma_{\alpha,\beta}$ has total curvature $-6\pi$.
\begin{figure}[htb]
	\begin{center}
		\includegraphics[height=4cm]{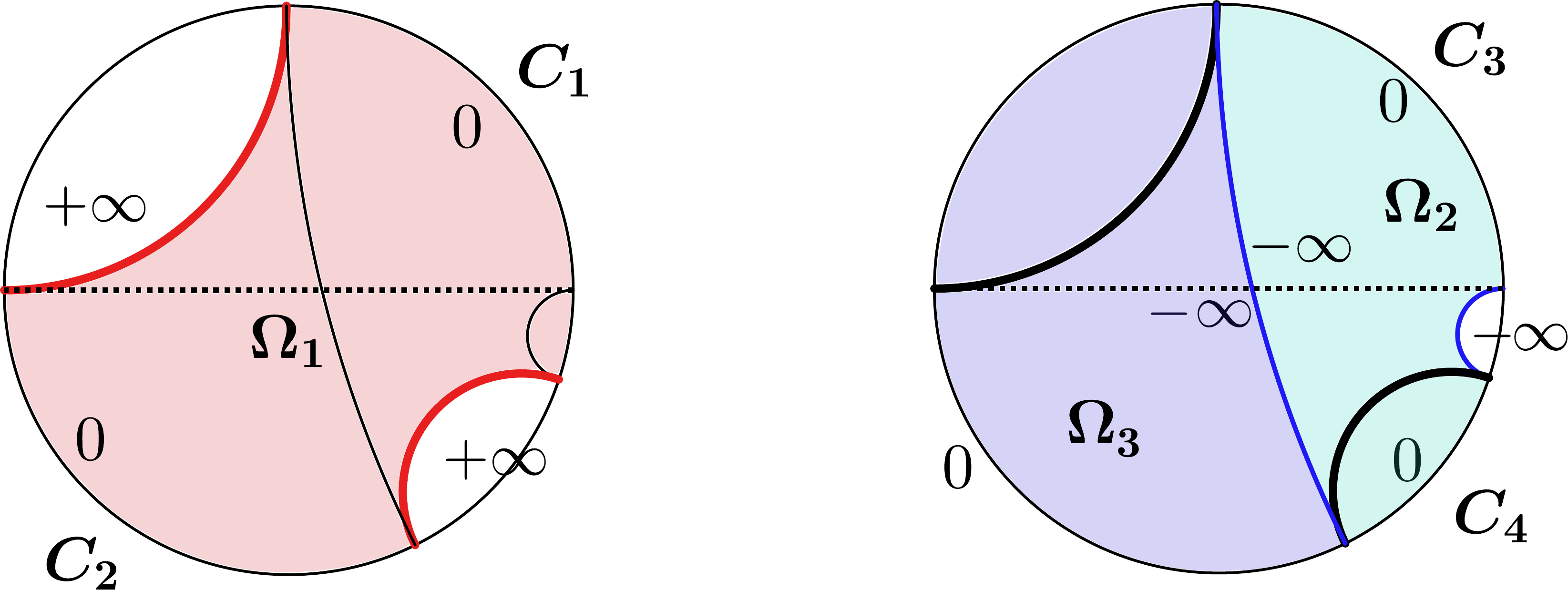}
	\end{center}
	\caption{The projections of the upper (red) and lower (blue and purple) barriers used in Proposition~\ref{prop: Ejemplo 2}}
	\label{Fig-barreras-2}
\end{figure}
\end{proof}
\begin{remark}
    It is expected that minimal surfaces of type \emph{(4)} in~\ref{th:configuraciones-2} with one vertical plane of symmetry  could be found as the conjugate sister of the surfaces $\bar\Sigma_{\alpha,\beta}$ of Proposition~\ref{prop: Ejemplo 2}. Note that the total curvature must be preserved by conjugation, and horizontal geodesics are sent to symmetry curves contained in vertical planes where the conjugate sister surface meets orthogonally.
\end{remark}


\begin{thebibliography}{}
	
	\bibitem{CM} J.\ Castro-Infantes, J.\ M.\ Manzano.
	\newblock Genus one minimal $k$-noids and saddle towers in $\mathbb{H}^2\times\mathbb{R}$.
	\newblock \emph{J.\ Inst.\ Math.\ Jussieu}, \textbf{22} (2023), no. 5, 2155–2175.

   
	
	\bibitem{CMR}
	J. {Castro-Infantes}, J. M. Manzano, M. Rodr{\'i}guez.
	\newblock A construction of constant mean curvature surfaces in
	{$\mathbb{H}^2\times\mathbb{R}$} and the {{Krust}} property.
	\newblock \emph{Int. Math. Res. Not. }, 2022, no. 19, 14605–14638.
	
		\bibitem{CS}
	J. {Castro-Infantes}, J. S. Santiago.
	\newblock Genus one $H$-surfaces with $k$-ends in $\mathbb H^2\times\R$
	\newblock Rev. Mat. Iberoam. \textbf{41} (2025), no. 1, 365–400.
	
	
	
	\bibitem{CR}
	P. Collin and H. Rosenberg,
	{\em Construction of harmonic diffeomorphisms and minimal graphs},
	Ann. of Math., {\bf 172} (2010),
	1879--1906. 
\bibitem{Coskunuzer} B. Coskunuzer,
     {\em Asymptotic plateau problem in  $\mathbb H^2\times\mathbb R$}, Selecta Math. (N.S.) {\bf24} (2018), no. 5, 4811–4838.
	
	
	
	
	\bibitem{HMR} 
	L. Hauswirth, A. Menezes and M.M. Rodr\'{\i}guez, 
	\newblock {\em On the characterization of minimal surfaces with finite total
		curvature in $\h^2\times\r$ and $\widetilde{\rm PSL}_2 (\r)$},
	\newblock Calc. Var. \textbf{58} (2019), no. 2, Paper No. 80, 24 pp.
   
	
	\bibitem{hnst}
	L. Hauswirth, B. Nelli, R. Sa Earp and  E. Toubiana,
	\newblock {\em Minimal ends in $\h^2\times\r$ with finite total
		curvature and a Schoen type theorem}, Advances in Mathematics, {\bf 274} (2015), 199--240.
	
	\bibitem{hr}
	L. Hauswirth and H. Rosenberg,
	\newblock {\em Minimal surfaces of finite total curvature in $\mathbb H\times\mathbb
		R$},
	\newblock Mat. Contemp., {\bf 31} (2006), 65--80  .
	
		
	\bibitem{hst}
	L. Hauswirth, R. Sa Earp and E. Toubiana,
	{\em Associate and conjugate minimal immersions in $M\times\bold
		R$}, Tohoku Math. J. (2) {\bf 60} (2008), 267-286.
	
	
	\bibitem{hu} 
	A. Huber,
	\newblock {\em On Subharmonic Functions and Differential Geometry in
		the Large},
	\newblock Comment. Math. Helvetic, {\bf 32} (1957), 181--206.

\bibitem{JM} 
	L. P. Jorge and W. H.  Meeks III
	\newblock {\em The Topology of Complete Minimal Surfaces of Finite Total Gaussian Curvature},
	\newblock Topology, {\bf 22} (1983), 203--221.Jorge

    
	\bibitem{MMR} 
	F. Mart\'\i n, R. Mazzeo and M.M. Rodr\'{\i}guez, 
	\newblock {\em Minimal surfaces with positive genus and finite total
		curvature in $\h^2\times\r$}, 
	\newblock Geometry and Topology, {\bf 18} (2014), 141--177. 
	
	\bibitem{MRR} L.\ Mazet, M. Rodr\'{i}guez, H.\ Rosenberg.
	\newblock The Dirichlet problem for the minimal surface equation --with possible infinite boundary data-- over domains in a Riemannian surface.
	\newblock \emph{Proc.\  London Math.\ Soc.\ (3)},  \textbf{102} (2011), no. 6, 985--1023.

    \bibitem{M-Y} W. H. Meeks III, S. T. Yau.
\newblock The existence of embedded minimal surfaces and the problem of uniqueness. \newblock \emph{Math. Z.} {\bf 179} (1982), no. 2, 151--168.

\bibitem{MPR} W. H. Meeks III, J. Pérez, A. Ros. 
\newblock Bounds on the topology and index of minimal surfaces. \newblock \emph{Acta Math.} {\bf 223 } (2019), no. 1, 113–149.


    
	\bibitem{MoR}
	F.~Morabito and M.M.~Rodr\'\i guez, {\it Saddle towers and minimal $k$-noids in
		$\mathbb{H}^2\times\mathbb{R}$},
	\newblock {J. Inst. Math. Jussieu}, {\bf 11} (2012), 333--349.

    \bibitem{NR} B.\ Nelli, H.\ Rosenberg.
	\newblock Minimal surfaces in $\mathbb{H}^2\times\mathbb{R}$.
	\newblock\emph{Bull.\ Braz.\ Math.\ Soc.}, \textbf{33} (2002), no. 2, 263-292.	
	
	\bibitem{oss} R. Osserman,
	{\it Global properties of minimal surfaces in $E^{3}$ and $E^{n}$"},
	Ann. Math. (2), {\bf 80} (1964), 340--364.
	
	
	
	\bibitem{p} J.~Pyo, {\it New complete embedded minimal surfaces in
		$\mathbb{H}^2\times\mathbb{R}$}, \newblock {Ann. Glob. Anal. Geom.},
	{\bf 40} (2011), 167--176.
	
	
	\bibitem{PR}
	J.~Pyo and M.M.~Rodr\'\i guez, {\it
		\newblock Simply-connected minimal surfaces with finite total curvature in
		$\mathbb{H}^2\times\mathbb{R}$}, Int. Math. Res. Notices, {\bf 2014} (2014), 2944-2954. 
	
	
	
	
	\bibitem{sch}  R. Schoen, {\em Uniqueness, symmetry, and embeddedness of minimal surfaces}, J. Diff.
	Geom., {\bf 18} (1983), 791--809.
	
\end{thebibliography}
\end{document}